\documentclass[12pt,reqno]{amsart}
\usepackage{amsmath}
\usepackage{}
\usepackage{}
\usepackage{amsfonts}
\usepackage{a4wide}
\usepackage{amsmath,amsthm,amssymb,amscd}
\usepackage{latexsym}
\usepackage{hyperref}
\usepackage[numbers,sort&compress]{natbib}
\usepackage{hypernat}
\allowdisplaybreaks
\numberwithin{equation}{section}

\newtheorem{theorem}{Theorem}[section]
\newtheorem{proposition}[theorem]{Proposition}

\newtheorem{lemma}[theorem]{Lemma}

\theoremstyle{definition}

\newtheorem{definition}[theorem]{Definition}

\newtheorem{remark}[theorem]{Remark}

\newcommand{\ds}{\displaystyle}

\def\R{\mathbb R}

\usepackage{color}

\begin{document}

\title[Fractional Schr\"{o}dinger Equations]
{Multi-peak semiclassical bound states  for Fractional Schr\"{o}dinger Equations with fast decaying potentials}

 \author{Xiaoming An\,\,\,and \,\,\,Shuangjie Peng }

 \address{School of Mathematics and Statistics\, \&\, Guizhou University of Finance and Economics, Guiyang, 550025, P. R. China}

\email{xman@mail.gufe.edu.cn}

\address{School of Mathematics and Statistics\, \&\, Hubei Key Laboratory of Mathematical Sciences, Central China
Normal University, Wuhan, 430079, P. R. China }

\email{sjpeng@mail.ccnu.edu.cn}

\begin{abstract}
We study the following fractional Schr\"{o}dinger equation
\begin{equation*}\label{eq0.1}
\varepsilon^{2s}(-\Delta)^s u + V(x)u = f(u), \,\,x\in\mathbb{R}^N,
\end{equation*}
 where $s\in(0,1)$. Under some conditions on $f(u)$, we show that the problem has a family of solutions concentrating at any finite given local minima of $V$ provided that $V\in C(\R^N,[0,+\infty))$. All decay rates of $V$ are admissible. Especially, $V$ can be compactly supported.  Different from the local case $s=1$ or the case of single-peak solutions, the nonlocal effect of the operator $(-\Delta)^s$ makes the peaks of the candidate solutions affect mutually, which causes more difficulties in finding solutions with multiple bumps. The methods in this paper are penalized technique and variational method.

\end{abstract}


\maketitle

\section{Introduction and main results}

In this paper, we consider the fractional Schr\"{o}dinger equation
\begin{equation}\label{eq1.1}
\varepsilon^{2s}(-\Delta)^s u + V(x)u = f(u),\,\,x\in\,\,\mathbb{R}^N,
\end{equation}
where $N > 2s$, $s\in (0,1)$, $V$ is a continuous function, $\varepsilon > 0$ is a small parameter, $f:\R^N\to\R$ is a nonlinear function. Problem~\eqref{eq1.1} is derived from the study of time-independent waves $\psi(x,t) = e^{-iEt}u(x)$ of the following nonlinear fractional Schr\"{o}dinger equation
$$
i\varepsilon\frac{\partial{\psi}}{{\partial t}} = \varepsilon^{2s}(-\Delta)^s{\psi} + U(x)\psi - f(\psi)\ x\in \mathbb{R}^N.
\eqno(NLFS)
$$
For example, letting $f(t) = |t|^{p - 2}t$, $V(x) = U(x) - E$ and inserting $\psi(x,t) = e^{-iEt}u(x)$ into $(NLFS)$, one can show that $(NLFS)$ is
\begin{equation}\label{MMMeq1.2}
\varepsilon^{2s}(-\Delta)^s u + V(x) u = |u|^{p-2}u.
\end{equation}

In physics, Eq \eqref{eq1.1} can be used to describe some properties of Einstein's theory of relativity and also has been derived as models of many physical phenomena, such as phase transition, conservation laws, especially in fractional quantum mechanics, etc., \cite{FQT}. $(NLFS)$ was introduced by Laskin (\cite{2}, \cite{3}) as an extension of the classical nonlinear Schr\"{o}dinger equations $s = 1$ in which the Brownian motion of the quantum
paths is replaced by a L\`{e}vy flight. To see more physical backgrounds, we refer to \cite{4}.

In this paper, we are interesting in semiclassical analysis of \eqref{eq1.1}. From a mathematical point of view, the transition from
quantum to classical mechanics can be formally performed by letting $\varepsilon\to 0$. For small $\varepsilon > 0$, solutions $u_{\varepsilon}$ are usually referred to as semiclassical bound states.

In the local case $s = 1$, the study of the nonlinear Schr\"odinger equation
$$
-\varepsilon^2\Delta u + V(x) u = f(u)
\eqno(NLS)
$$
has been extensively investigated in the semiclassical regime and a considerable amount of
work has been done, showing that existence and concentration phenomena of single- and
multi-bump solutions occur at critical points of the electric potential $V$ when $\varepsilon\to 0$, see \cite{{5},{6},{7},{8},{9},{10},{11},{12},{13},{14},{15},{16},{17},{BDP},{CN}} and the references therein for example.

To our best knowledge, there are few results on the semiclassical bound states to problem \eqref{eq1.1} in the nonlocal case $s\in (0,1)$.
Basing on the well-known non-degenerate results in \cite{{19},{20}} and the mathematical reduction method,  it was proved  in  \cite{{25},{26},{27}} that  problem \eqref{MMMeq1.2} has solutions concentrating at the prescribed non-degenerate critical points of $V$ when $\varepsilon\to 0$.
When $\inf_{x\in \mathbb{R}^N}V(x) > 0$ and $V$ has local minimum which may be degenerate, Alves et al. in \cite{100} used the penalized method developed by del Pino et al. in \cite{10} and the extension method developed by Caffarelli et al. in \cite{24} to construct solutions concentrating at a  local minimum of $V$ when $\varepsilon\to 0$.  Successively, assuming more weakly that $\liminf_{|x|\to\infty}V(x)|x|^{2s} \ge 0$, in \cite{APX,ADP-AA-2021}, solutions concentrating at a local minimum of $V$  were also obtained. We point out here that the solutions found in \cite{100} and \cite{APX} have exactly one local maximum and hence  are single-peaked.

However, concerning \eqref{eq1.1}, up to now there are no research on the multi-bump solutions in the case that the potentials $V(x)$ vanish at infinity and  critical points of $V(x)$ are degenerate. The main difficulty lies in that for a suitable function $u:\ \R^N\to\R$, under the nonlocal effects of $(-\Delta)^s$, one can not compute $(-\Delta )^su$  as precisely  as $-\Delta u$. Moreover, the nonlocal operator $(-\Delta)^s$ makes the peaks of the candidate solutions affect mutually, which causes more difficulties in finding solutions  with multiple bumps (see the estimates of \eqref{MMMeq2.26}, \eqref{MMMeq2.28} and \eqref{MMMeq2.29} in Lemma \ref{leA.2} for example).

This paper devotes to finding  solutions with multiple bumps for more general potentials including fast decaying potentials, i.e.,
$$
\liminf_{|x|\to\infty}V(x)|x|^{2s}=0,
$$
in which, a typical  case  is  that  $V$ is compactly supported.

In order to state our main result, we need to introduce  some notations and assumptions. For $s\in(0,1)$, the fractional Sobolev space $H^s(\mathbb{R}^N)$ is defined as
$$
H^s(\mathbb{R}^N) = \Big\{u\in L^2(\mathbb{R}^N):\frac{u(x) - u(y)}{|x - y|^{N/2 + s}}\in L^2(\mathbb{R}^N\times\mathbb{R}^N)\Big\},
$$
endowed with the norm
$$
\|u\|_{H^s(\mathbb{R}^N)} = \Big(\int_{\mathbb{R}^N}|(-\Delta)^{s/2}u|^2 + u^2\,  {\mathrm{d}}x \Big)^{\frac{1}{2}},
$$
where
$$
\int_{\mathbb{R}^N}|(-\Delta)^{s/2}u|^2{\mathrm{d}}x  = \int_{\mathbb{R}^{2N}}\frac{|u(x) - u(y)|^2}{|x - y|^{N + 2s}}\,{\mathrm{d}}x \,  {\mathrm{d}}y.
$$
Like the classical case, we define the space $\dot{H}^s(\R^N)$ as the completion of $C^{\infty}_c(\R^N)$ under the norm
$$
\|u\|^2 = \int_{\mathbb{R}^N}|(-\Delta)^{s/2}u|^2{\mathrm{d}}x  = \int_{\mathbb{R}^{2N}}\frac{|u(x) - u(y)|^2}{|x - y|^{N + 2s}}\,{\mathrm{d}}x \,  {\mathrm{d}}y.
$$
Define the following fractional Sobolev space
$$
W^{s,2}(\Omega) = \Big\{u\in L^2(\Omega): \frac{|u(x) - u(y)|}{|x - y|^{\frac{N}{2} + s}}\in L^2(\Omega\times\Omega)\Big\}.
$$
It is easy to check that with the inner product
$$
(u,v) = \int_{\Omega}\int_{\Omega}\frac{(u(x) - u(y))(v(x) - v(y))}{|x - y|^{N + 2s}}dxdy + \int_{\Omega}uvdx\ \ \forall u,v\in W^{s,2}(\Omega),
$$
$W^{s,2}(\Omega)$ is a Hilbert space (see \cite{4} for details). According to \cite{4}, the fractional Laplacian   is defined as
\begin{align*}
(-\Delta)^s u(x) \, &= \, C  (N,s)P.V.\int_{\mathbb{R}^N}\frac{u(x) - u(y)}{|x - y|^{N + 2s}}\,{\mathrm{d}}y
\\[2mm]
\, & =\,   C(N,s)\lim_{\varepsilon\to 0}\int_{\mathbb{R}^N          \backslash        B_{\varepsilon}(x)      }\frac{u(x) - u(y)}{|x - y|^{N + 2s}}\,{\mathrm{d}}y.
\end{align*}
For the sake of simplicity, we define for every $u\in \dot{H}^s(\R^N)$ the fractional $(-\Delta)^su$ as
$$
(-\Delta)^s u(x) = \int_{\mathbb{R}^N}\frac{u(x) - u(y)}{|x - y|^{N + 2s}}\,{\mathrm{d}}y.
$$
Our solutions will be found in the following weighted fractional Sobolev  space:
$$
\mathcal{D}^s_{V,\varepsilon}(\mathbb{R}^N) = \left\{u\in \dot{H}^s(\R^N):\,\,u\in L^2\big(\mathbb{R}^N,V(x)\,{\mathrm{d}}x\big)\right\},
$$
endowed with the norm
$$
\|u\|_{\mathcal{D}^s_{V,\varepsilon}(\mathbb{R}^N)} = \Big(\int_{\mathbb{R}^N}\varepsilon^{2s}|(-\Delta)^{s/2}u|^2 + Vu^2\,{\mathrm{d}}x \Big)^{\frac{1}{2}}.
$$

For the nonlinear term $f(u)$, we assume
\begin{align}\label{Meq1.6}
\nonumber&(f_1)\ f(t)\ \text{is an odd function and}\ f(t) = o(t^{1 + \tilde{\kappa}})\ \text{as}\ t\to 0^+,\ \text{where}\ \tilde{\kappa} = \frac{2s + 2\kappa}{N-2s -\tilde{\nu}}>0\\
 \nonumber&\qquad\text{with}\ \tilde{\nu},\kappa>0\ \text{are  small parameters}.\\
&(f_2)\ \lim_{t\to\infty}\frac{f(t)}{t^p} = 0\  \text{for some}\ 1<p<2^*_{s} - 1.\\
\nonumber&(f_3)\ \text{There exists}\ 2<\theta \leq p + 1\ \text{such that}\
0\leq\theta F(t) < f(t)t\ \text{for all}\ t>0,\
\text{where}\\
\nonumber&\qquad F(t) = \ds\int_{0}^tf(\alpha)d\alpha.\\
\nonumber&(f_4)\ \text{The map}\ t\mapsto \frac{f(t)}{t}\ \text{is increasing on}\ (0,+\infty).
\end{align}

A typical case of $f(t)$ is:\,$f(t)=|t|^{p-2}t$ with $2+\frac{2s}{N-2s}<p<2^*_s$.

For the potential term $V$, we assume that $V\in C\big(\mathbb{R}^N,[0,\infty)\big)$ and

$(\textbf{V})$  There exist open bounded sets $\Lambda_i\subset\subset S_i\subset\subset U_i$ with smooth boundaries, such that
\begin{equation}\label{{Aeq1.3}}
0 < \lambda_i = \inf_{\Lambda_i}V < \inf_{U_i\backslash\Lambda_i}V,\ \overline{U_i}\cap \overline{U_j} = \emptyset\ \text{if}\ 1\leq i\neq j\leq k.
\end{equation}
Denote $\Lambda = \bigcup_{i = 1}^k\Lambda_i$, $S = \bigcup_{i = 1}^kS_i$
and $U = \bigcup_{i = 1}^kU_i$. Without loss of generality, we assume that $0\in \Lambda$.

\begin{theorem}\label{th1.1}
Let $N > 2s$, $s\in(0,1)$,  $V$ satisfy (\textbf{V}) and $f$ satisfy the assumptions $(f_1)-(f_4)$. Then problem \eqref{eq1.1} has a  positive solution $u_{\varepsilon}\in \mathcal{D}^s_{V,\varepsilon}(\mathbb{R}^N)$ if $\varepsilon>0$ is small enough. Moreover, there exists $k$ families of points $\{\{x^i_{\varepsilon}\}:1\le i\le k\}$ and an $\alpha$ close to $N - 2s$, such that
\begin{align*}
  &(i)\ \ {\lim\limits_{\varepsilon\to 0}V(x^i_{\varepsilon})} = \lambda_i,\\
  &(ii)\ \ \liminf_{\varepsilon\to 0}\|u_{\varepsilon}\|_{L^{\infty}\big(B_{\varepsilon\rho}(x^i_{\varepsilon})\big)}>0\\
  &(iii)\ \ u_{\varepsilon}(x)\le \sum_{i=1}^k\frac{C\varepsilon^{\alpha}}{\varepsilon^{\alpha} + |x - x^i_{\varepsilon}|^{\alpha}},
\end{align*}
where $C$ and $\rho$ are positive constants.
\end{theorem}

Now we introduce the main idea of the proof. For the local case $s=1$, certain penalized functional like
\begin{equation}\label{eq1.2}
K_{\varepsilon}(u) = M_1\sum_{j = 1}^k\Big(\big((L^j_{\varepsilon}(u))^{1/2}_+ - \varepsilon^{N/2}(c_j + \sigma_j)^{1/2}\big)_+\Big)^2
\end{equation}
was usually  employed  to prove that the penalized solution $u_{\varepsilon}$ has exactly  one peak in each $\Lambda_i$, see \cite{BDP,MM3,CN} for example.  But, the key step of this argument is to eliminate the effect of
$K_{\varepsilon}(u)$ to the equation, which   needs  a type of  isolated property of the least energy of $-\Delta u + u = g(u)$. However, for our case $0<s<1$, this type of isolated property is still unknown. To overcome this difficulty, we use the method developed by Byeon and Jeanjean in \cite{J.B-L.J-DCDS-2007}, which proves the existence of multi-peak solutions of  following equation
\begin{equation}\label{Meq1.4}
-\varepsilon^2\Delta u + V(x) u = g(u)
\end{equation}
by using only the compactness of the set consisting of the radial positive least energy solutions of the following limiting problem of \eqref{Meq1.4}:
\begin{equation*}
-\Delta u + a u = g(u),
\end{equation*}
where $a>0$ is a constant and $g$ is a nonlinear term satisfying some subcritical conditions. {For more application of this methods,
see \cite{T.Hu-W.Shuai-JDE-2018}.} Roughly speaking, by the compact property, we use the deformation ideas of Lemma 2.2 in \cite{MW} to construct a $(PS)_c$ sequence near the least energy solutions of the following $k$ problems:
$$
    (-\Delta)^s u + \lambda_i u = f(u),  \text{in}\ \R^N,\,\,i=1,\cdots,k.
$$

It is worth mentioning that the compact property   can be obtained by the decay estimates of positive radial least energy solutions(see Proposition \ref{Mpr2.5} below). However, the vanishing of $V$ and the nonlocal effect of $(-\Delta)^s$ makes the construction of multi-peak solutions more difficult than the classical case $s = 1$, the {non-vanishing case(\cite{J.B-L.J-DCDS-2007})} and the single peak case \cite{ADP-AA-2021}. Firstly,
an elementary(but tedious) calculations show that when $V(x)$ vanishes faster than $|x|^{-2s}$, the natural functional $I_{\varepsilon}:\mathcal{D}^s_{V,\varepsilon}(\mathbb{R}^N)\to \mathbb{R}$ corresponding to \eqref{eq1.1} defined as
$$
I_{\varepsilon}(u) = \frac{1}{2}\int_{\mathbb{R}^N}(\varepsilon^{2s}|(-\Delta)^{s/2}u|^2 + Vu^2)dx - \int_{\mathbb{R}^N}F(u)dx,
$$
whose critical points are solutions of equation \eqref{eq1.1}, is not well-defined in $\mathcal{D}^s_{V,\varepsilon}(\mathbb{R}^N)$, where $F(t)=\int_{0}^tf(s)ds$. Moreover, the fact that $V(x)$ may be compactly supported  makes it impossible that $V$ can  dominated the nonlinear term $|u|^{p-2}u$ like \cite{APX}.
Hence we have to introduce a different penalized idea from {\cite{J.B-L.J-DCDS-2007}} to cut-off the nonlinear term. More precisely,  we will first  use the nonlocal part $(-\Delta)^s$ to modify the problem by the following fractional Hardy inequality
\begin{equation}\label{eq1.4}
\int_{\R^N}\frac{|u(x)|^2}{|x|^{2s}}\,{\mathrm{d}}x\le C_{N,s}\|(-\Delta)^{s/2}u\|^2_2
\end{equation}
for all $u\in \dot{H}^s(\R^N)$(see \cite{FS}), and then construct a sup-solution and estimate the energy   of  multi-peak solutions.

The celebrated paper \cite{24} provides an easy way to understand the nonlocal problem (see \cite{100} for example), by which, one can convert the nonlocal problem \eqref{eq1.1} into a local problem. But we do not use this method in our paper. Indeed, if  problem \eqref{eq1.1} becomes a local problem, the vanishing of $V$ and the added variable ``$t>0$" (which comes from extending the problem into $\mathbb{R}^{N + 1}_+$, see \cite{100} for instance) will make it difficult to construct precise penalized functions.

%Finally, we want to emphasize that, since we only need the set of positive radial least energy solutions of $(-\Delta)^s u + au=|u|^{p-2}u\ \in\R^N\ (a>0)$(see \eqref{Meq2.9}) to be compact, our method also works well to problem \eqref{eq1.1} with general nonlinearity. For example, the same results like Theorem \ref{th1.1} can be obtained if $f:\mathbb{R}\to \mathbb{R}$ is assumed to satisfy the following properties(see Section \ref{s5} for more details):
%\begin{equation}\label{eq5.1}
%  \varepsilon^{2s}(-\Delta)^su + V(x)u = f(u),
%\end{equation}
%where the potential $V(x)$ is the same as before, the nonlinearity $f:\mathbb{R}\to \mathbb{R}$ is assumed to satisfy the following properties:
%\begin{align}\label{Meq1.6}
%\nonumber&(f_1)\ f\ \text{is an odd function and}\ f(t) = o(t^{1 + \frac{2\sigma}{N - 2s - \tilde{\delta}}}) \text{as}\ t\to 0^+\ \text{where}\ \tilde{\delta}>0\\
%\nonumber&\qquad \text{is a small parameter}.\\
%&(f_2)\ \lim_{t\to\infty}\frac{f(t)}{t^p} = 0\  \text{for some}\ 1<p<2^*_{s} - 1.\\
%\nonumber&(f_3)\ \text{There exists}\ 2<\theta \leq p + 1\ \text{such that}
%0\leq\theta F(t) < f(t)t\ \text{for all}\ t>0\
%\text{where}\\
%\nonumber&\qquad F(t) = \ds\int_{0}^tf(\alpha)d\alpha.\\
%\nonumber&(f_4)\ \text{The map}\ t\mapsto \frac{f(t)}{t}\ \text{is increasing on}\ (0,+\infty).
%\end{align}
\vspace{0.33cm}

  The paper is organized as follows: in Section 2, we establish the penalized scheme. By using the compact property of the set consisting  of positive radial least energy solutions and the deformation idea
   in Lemma 2.2 of \cite{MW}, we construct a $(PS)_c$ sequence with $k$-peaks in $\Lambda$, and then get a penalized multi-peak solution. In Section \ref{s4}, we construct a penalized function to prove
    that the penalized solution is indeed a solution of the original equation \eqref{eq1.1}. In the Appendix we will give some tedious energy estimates caused by the nonlocal operator.
%In Section \ref{s5}, we gill give a short proof to Theorem \ref{th1.1} when the nonlinear term is replaced by $f$ that satisfies \eqref{Meq1.6}.

\section{The penalized problem}

In this section, we first  establish a penalized problem  by using the fractional Hardy inequality \eqref{eq1.4} to cut off  the nonlinear term $f$. A well-defined smooth penalized functional in $\mathcal{D}^s_{V,\varepsilon}(\R^N)$ will be obtained. Secondly, we use the compact property of set consisting  of least energy solutions and the deformation lemma \cite[Lemma 2.2]{MW} to construct a $(PS)_c$ sequence near the least energy solutions. A penalized solution with $k$ peaks for the penalized problem will be obtained by passing limit on the $(PS)_c$ sequence.

%first use the fractional Hardy inequality to modify the origin problem, from which we get a well defined  and smooth penalized functional in $\mathcal{D}^s_{V,\varepsilon}(\R^N)$. Secondly, using the compact of set of least energy solutions, we use deformation lemma to construct a $(PS)_c$ sequence near the $k$ sets of least energy solutions for the penalized functional. Finally, by passing limit on the $(PS)_c$ sequence, we get a penalized solution with multi-peak for the penalized problem.

\vspace{0.5cm}

\subsection{The Penalized Functional}

\noindent The following inequality exposes the relationship between $H^s(\mathbb{R}^N)$ and the Banach space $L^q(\mathbb{R}^N)$.
\begin{proposition}\label{pr2.1} (Fractional version of the Gagliardo$-$Nirenberg inequality)(\cite{M.I.-Weinstein-CPDE-1987})
For every $u\in H^s(\mathbb{R}^N)$,
$$
\|u\|_{q} \leq C \|(-\Delta)^{s/2}u\|^{\beta}_{2}\|u\|^{1 - \beta}_{2},
$$
where $q\in [2,2^*_s]$ and $\beta$ satisfies $\frac{\beta}{2^*_s} + \frac{(1 - \beta)}{2} = \frac{1}{q}$.
\end{proposition}

The above inequality implies that $ H^s(\mathbb{R}^N)$ is continuously  embedded into $L^q(\mathbb{R}^N)$ for $\ q\in [2,2^*_s]$. Moreover, on bounded set, the embedding is compact ( see \cite{4}), i.e.,
$$
H^s(\mathbb{R}^N)\subset\subset L^q_{loc}(\mathbb{R}^N)\ \text{compactly, if}\ q\in [1,2^*_s).
$$

\subsection{The penalized functional}

Now we are going to modify the original problem \eqref{eq1.1}.
According to the fractional Hardy inequality \eqref{eq1.4}, we choose a family of  penalized potentials $\mathcal{P}_{\varepsilon}\in L^{\infty}(\mathbb{R}^N,[0,\infty))$ for $\varepsilon > 0$ small in such a way that
\begin{eqnarray}\label{eq2.1}
\begin{split}
\left\{
  \begin{array}{ll}
    \mathcal{P}_{\varepsilon}(x) = 0,\ x\in\Lambda,&\\
    \lim\limits_{\varepsilon \to 0}\sup_{\mathbb{R}^N\backslash \Lambda} \mathcal{P}_{\varepsilon}(x)\varepsilon^{-(2s + 3\kappa/2)}|x|^{2s + \kappa} = 0,&
  \end{array}
\right.
\end{split}
\end{eqnarray}
%{and}
%\begin{align}
%\ \lim\limits_{\varepsilon \to 0}\sup_{\mathbb{R}^N\backslash \Lambda} \mathcal{P}_{\varepsilon}(x)\varepsilon^{-(2s + 3\kappa/2)}|x|^{2s + \kappa} = 0,
%\end{align}
where $\kappa>0$ is the same parameter in $(f_1)$. Noting that by \eqref{eq1.4}, when $\varepsilon>0$ is small enough, it holds that for any $A\subset \R^N$,
\begin{align}\label{MMMeq2.2}
\int_{A}\mathcal{P}_{\varepsilon}(x)|u|^2\le C_{N,s}\frac{\varepsilon^{2s +\frac{3\kappa}{2}}}{\inf_{x\in(\R^N\backslash\Lambda)\cap A}{|x|^{\kappa}}}\int_{\R^N}|(-\Delta)^{s/2}u|^2\ \ \text{for all}\ u\in \mathcal{D}^s_{V,\varepsilon}
\end{align}
where $C_{N,s}$ is the constant in \eqref{eq1.4}. This type of estimate plays a key role in the paper(see \eqref{Beq2.10} below for example).

Now we give the penalized problem according to the choice of $\mathcal{P}_{\varepsilon}$:
\begin{equation}\label{eq2.2}
\varepsilon^{2s}(-\Delta)^su + Vu = \chi_{\Lambda}f(s_+) + \chi_{\mathbb{R}^N\backslash\Lambda}\min\{f(s_+),\mathcal{P}_{\varepsilon}(x)s_+\}.
\end{equation}
It is easy to check that if a solution $u_{\varepsilon}$ of \eqref{eq2.2} satisfies
$$
f(u_{\varepsilon})\le \mathcal{P}_{\varepsilon}u_{\varepsilon}\ \ \text{on}\ \R^N\backslash\Lambda,
$$
then $u_{\varepsilon}$ is a solution of \eqref{eq1.1}.

%By the explicit decay rates $(i)$ of $P_{\varepsilon}$, we have:
%\begin{proposition}\label{pr2.1} (\cite[Lemma 4.2]{APX})
%The embedding $\mathcal{D}^s_{V,\varepsilon}(\mathbb{R}^N)\subset\subset L^2(\mathbb{R}^N,(P_{\varepsilon}(x) + \chi_{\Lambda}(x)))$ is compact. Moreover, for every $\delta>0$, there exits a $\varepsilon_{\delta} > 0$ such that if $0<\varepsilon<\varepsilon_{\delta}$, then
%$$
%\int_{\mathbb{R}^N}P_{\varepsilon}(x)|\varphi|^2  \leq \delta \int_{\mathbb{R}^N}\varepsilon^{2s}|(-\Delta)^{s/2} \varphi|^2 + V|\varphi|^2,
%$$
%for all $\varphi\in \mathcal{D}^s_{V,\varepsilon}(\mathbb{R}^N)$.
%\end{proposition}

    Given a penalized potential $\mathcal{P}_{\varepsilon}$ that satisfies \eqref{eq2.1},
we define the penalized nonlinearity $g_{\varepsilon}:\mathbb{R}^N\times \mathbb{R}\to \mathbb{R}$ as
$$
g_{\varepsilon}(x,s): = \chi_{\Lambda}f(s_+) + \chi_{\mathbb{R}^N\backslash\Lambda}\min\{f(s_+),\mathcal{P}_{\varepsilon}(x)s_+\}.
$$
We denote $G_{\varepsilon}(x,t) = \int_{0}^{t}g_{\varepsilon}(x,s)ds$.

Accordingly, the penalized superposition operators $\mathfrak{g}_{\varepsilon}$ and $\mathfrak{G}_{\varepsilon}$ are given by
$$
\mathfrak{g}_{\varepsilon}(u)(x) = g_{\varepsilon}(x,u(x))\ \text{and}\ \mathfrak{G}_{\varepsilon}(u)(x) = G_{\varepsilon}(x,u(x)).
$$
Following, we define the penalized functional $J_{\varepsilon}:\mathcal{D}^s_{V,\varepsilon}(\mathbb{R}^N)\to \mathbb{R}$ as
$$
J_{\varepsilon}(u) = \frac{1}{2}\int_{\mathbb{R}^N}(\varepsilon^{2s}|(-\Delta)^{s/2} u|^2 + V(x)|u|^2) - \int_{\mathbb{R}^N}\mathfrak{G}_{\varepsilon}(u).
$$

The strong assumption \eqref{eq2.1} can help to check that $J_{\varepsilon}$ is $C^1$ and satisfies (P.S.) condition.

\begin{lemma}\label{le2.4}

(1) If $2 < p< 2^*_s$ and \eqref{eq2.1} hold, then $J_{\varepsilon}\in C^1(\mathcal{D}^s_{V,\varepsilon}(\mathbb{R}^N), \mathbb{R})$ and for $u\in \mathcal{D}^s_{V,\varepsilon}(\mathbb{R}^N)$, $\varphi\in \mathcal{D}^s_{V,\varepsilon}(\mathbb{R}^N)$,
$$
\langle J_{\varepsilon}'(u),\varphi\rangle = \int_{\mathbb{R}^N}\varepsilon^{2s}(-\Delta)^{s/2}u(-\Delta)^{s/2}\varphi + Vu\varphi - \int_{\mathbb{R}^N}\mathfrak{g}_{\varepsilon}(u)\varphi .
$$

Here $\langle\cdot,\cdot\rangle$ denotes the duality product between the dual space $\mathcal{D}^s_{V,\varepsilon}(\mathbb{R}^N)'$ and the space $\mathcal{D}^s_{V,\varepsilon}(\mathbb{R}^N)$. In particular, $u\in \mathcal{D}^s_{V,\varepsilon}(\mathbb{R}^N)$ is a critical point of $J_{\varepsilon}$ if and only if $u$ is a weak solution of the penalized equation
\begin{equation}\label{BBBeq2.2}
\varepsilon^{2s}(-\Delta)^s u + Vu =\mathfrak{g}_{\varepsilon}(u).
\end{equation}

(2) ((P.S.) condition) If $2 < p <2^*_s$  and \eqref{eq2.1} holds, then $J_{\varepsilon}$ owns the mountain pass geometry and satisfies the Palais-Smale condition.
\end{lemma}

\begin{proof}
We omit the proof since it is quite similar to that in \cite[Lemma 2.4]{ADP-AA-2021}.
\end{proof}

\subsection{Construction of  solutions with $k$ peaks}

\begin{definition}\label{de2.6}
For  $a>0$, we define the value $c_a$ as
$$
c_a = \inf_{\gamma\in\Gamma_a}\max_{t\in [0,1]}L_a(\gamma(t)),
$$
where $L_{a}:H^s(\mathbb{R}^N)\to \mathbb{R}$ and $\Gamma_a$ are given by
$$
L_a(u) = \frac{1}{2}\int_{\mathbb{R}^N}\int_{\mathbb{R}^N}\frac{|u(x) - u(y)|^2}{|x - y|^{N + 2s}}dy + \frac{1}{2}\int_{\mathbb{R}^N}a|u|^2 - \int_{\mathbb{R}^N}F(u)
$$
and
\begin{equation*}
\Gamma_a : = \{\gamma\in (C[0,1],H^s(\mathbb{R}^N)):\gamma(0) = 0,\ L_a(\gamma(1)) < 0\},
\end{equation*}
where $F(t) =\int_{0}^tf(s)ds$. From \cite{S,FQT}, we know that $c_{a}$ is continuous, increasing on $a$ and can be achieved by a positive radial solution $U_a$ which satisfies the following limiting problem
$$
(-\Delta)^su + au = f(u),\,\,x\in\R^N.
$$
Moreover, there exist two positive constants $\tilde{c}_a,\ \widetilde{C}_a$  such that
\begin{equation}\label{MMMeq2.5}
\frac{\tilde{c}_a}{1+|x|^{N+2s}}\le  U_a(|x|)\le \frac{\widetilde{C}_a}{1+|x|^{N+2s}},\ \ x\in\R^N.
\end{equation}
Then, letting $S_a=\{U_a:U_a\ \text{is positive radial and achieves}\ c_a\}$, by the decay estimate \eqref{MMMeq2.5},  we have
\begin{proposition}\label{Mpr2.5}
The set $S_a$ is compact in $H^s(\R^N)$.
\end{proposition}
\begin{proof}
If $S_a$ contains finitely many elements, then it is compact. Otherwise, taking a sequence $\{U_n\}\subset S_a$, since $\{U_n\}$ is bounded in $H^s(\R^N)$, there exists a $\overline{U}\in H^s(\R^N)$ such that
$$
\left\{
  \begin{array}{ll}
    U_n\rightharpoonup\ \overline{U}\text{weakly in}\ H^s(\R^N), & \\
    U_n\to\ \overline{U}\ \text{a.e.}\ \text{in}\ \R^N, & \\
    U_n\to\ \overline{U}\ \text{strongly in}\ L^q_{loc}(\R^N),\ 1<q<2^*_s - 1. &
  \end{array}
\right.
$$
Then, by \eqref{MMMeq2.5}, we have $U_n\to\ \overline{U}\ \text{strongly in}\ L^p(\R^N)$. Obviously, $\overline{U}$ is nonnegative and satisfies
$$
(-\Delta)^s \overline{U} + a\overline{U} =  f(\overline{U}).
$$
Furthermore, by standard regularity argument(see Appendix D in \cite{20} for example), we have $\overline{U}>0$. Then, by Definition \ref{de2.6}, we have $\liminf_{n\to\infty}L_{a}(U_n)\ge L_{a}(\overline{U})\ge c_a$. Then $L_{a}(\overline{U})= c_a$, $\overline{U}\in S_a$ and
$$
\int_{\R^N}|(-\Delta)^{s/2} U_n|^2 + a|U_n|^2 \to  \int_{\R^N}|(-\Delta)^{s/2} \overline{U}|^2 + a|\overline{U}|^2
$$
as $n\to\infty$. This completes the proof.
\end{proof}
\end{definition}

From now on we define
$$
\mathcal{M}_i = \{x\in\Lambda_i:V(x)=\lambda_i\}\ \text{and}\ \mathcal{M}=\bigcup_{i=1}^k\mathcal{M}_i.
$$
Let $\eta(x)=\eta(|x|)\in C^{\infty}_c(\R^N)$ satisfy $0\le \eta\le 1$, $\eta \equiv 1$ on $\overline{B}_{\beta}(0)$ and $\eta\equiv 0$ on $\R^N\backslash B_{2\beta}(0)$, where $\beta>0$ is a small parameter satisfying $\mathcal{M}^{2\beta}\subset \Lambda$. For each {$p_i\in\mathcal{M}_i$} and $U_{\lambda_i}\in S_{\lambda_i}$ given by Definition \ref{de2.6}, we define
$$
U^{p_1,\ldots,p_k}_{\varepsilon}(x) = \sum_{i=1}^k\eta(x - p_i)U_{\lambda_i}\Big(\frac{x - p_i}{\varepsilon}\Big),\ \ x\in\R^N.
$$
We will find a solution to \eqref{BBBeq2.2}, for sufficiently small $\varepsilon>0$, near the set
$$
\mathcal{X}_{\varepsilon}=\left\{U^{p_1,\ldots,p_k}_{\varepsilon}:U_{\lambda_i}\in S_{\lambda_i},\ p_i\in\mathcal{M}_i,\ 1\le i\le k\right\}.
$$
For each $1\le i\le k$, we also define
$$
W^i_{\varepsilon}(x)=\eta(x - p_i)U_{\lambda_i}\Big(\frac{x-p_i}{\varepsilon}\Big).
$$
We have:
\begin{proposition}\label{Mpr2.6}
For each $i\in \{1,\ldots,k\}$, it holds
$$
J_{\varepsilon}\big(\sum_{j=1}^kt_jW^j_{\varepsilon})<0
$$
if $t_i>T$ for some $T\in(0,+\infty)$.
\end{proposition}
\begin{proof}
By the choice of $W^i_{\varepsilon}$, there exists a positive constant $C$ such that
\begin{align*}
J_{\varepsilon}\big(\sum_{i=1}^kt_iW^i_{\varepsilon})
&=  \sum_{{i = 1,k=1}\atop{i\ne j}}^{k}\frac{\varepsilon^{2s}}{2}\int_{\R^N\times\R^N}\frac{t_it_j(W^i_{\varepsilon}(x) - W^i_{\varepsilon}(y))(W^j_{\varepsilon}(x) - W^j_{\varepsilon}(y))}{|x - y|^{N + 2s}}dxdy\\
&\quad + \sum_{i = 1}^{k}\Big(\frac{t^2_i}{2}\|W^i_{\varepsilon}\|^2_{\varepsilon} - \int_{\R^N}F(t_iW^i_{\varepsilon})\Big)\\
&\le \sum_{i = 1}^{k}\Big(Ct^2_i\|W^i_{\varepsilon}\|^2_{\varepsilon} - \int_{\R^N}F(t_iW^i_{\varepsilon})\Big)\\
&=\varepsilon^N\sum_{i = 1}^{k}\Big(Ct^2_i\|\eta_{\varepsilon}(x)U_{\lambda_i}(x)\|^2 - \int_{\R^N}F(t_i\eta_{\varepsilon}(x)U_{\lambda_i}(x))\Big).
\end{align*}
By decomposition, we have
\begin{eqnarray}\label{MMMMeq2.5}
\begin{split}
&\quad\|\eta_{\varepsilon}(x)U_{\lambda_i}(x)\|^2\\
&=\|U_{\lambda_i}(x)\|^2 + \int_{\R^N\times\R^N}\frac{(\eta^2_{\varepsilon}(x)-1)
|U_{\lambda_i}(x)-U_{\lambda_i}(y)|^2}{|x - y|^{N+2s}}dxdy\\
&\quad+\sum_{i=1}^kt^2_i\int_{\R^N}\int_{\R^N}\frac{\eta_{\varepsilon}(x)(U_{\lambda_i}(x)-U_{\lambda_i}(y))(\eta_{\varepsilon}(x)-\eta_{\varepsilon}(y))U_{\lambda_i}(y)}{|x-y|^{N+2s}}dxdy\\
&\quad+\int_{\R^N}\int_{\R^N}\frac{(\eta_{\varepsilon}(x)-\eta_{\varepsilon}(y))^2U^2_{\lambda_i}(y)}{|x-y|^{N+2s}}dxdy.
\end{split}
\end{eqnarray}
But, arguing as done in the proof of the following \eqref{MMMeq2.26}, \eqref{MMMeq2.28} and \eqref{MMMeq2.29} in Lemma \ref{leA.2}, we know that
$$
\int_{\R^N}\int_{\R^N}\frac{(\eta_{\varepsilon}(x)-\eta_{\varepsilon}(y))^2U^2_{\lambda_i}(y)}{|x-y|^{N+2s}}dxdy=o_{\varepsilon}(1).
$$
Hence
\begin{align*}
J_{\varepsilon}\big(\sum_{i=1}^kt_iW^i_{\varepsilon})
&\le \varepsilon^N\sum_{i = 1}^{k}\Big(Ct^2_i\|U_{\lambda_i}(x)\|^2 - \int_{B_1(0)}F(t_iU_{\lambda_i}(x))\Big).
\end{align*}
Then, by the assumption on $f$ and $\max\limits_{{t>0}\atop{1\le i\le k}}\big(Ct^2_i\|U_{\lambda_i}(x)\|^2 - \int_{B_1(0)}F(t_iU_{\lambda_i}(x))\big)<+\infty$, we get the conclusion.
\end{proof}

As a result of Proposition \ref{Mpr2.6}, we know that the following definition is reasonable: for $\tau=(t_1,\ldots,t_k)\in[0,T]^k$,  let $\gamma_{\varepsilon}(\tau)=\sum_{i=1}^kt_iW^i_{\varepsilon}$ and define
$$
\mathcal{D}_{\varepsilon}=\max_{\tau\in [0,T]^k}J_{\varepsilon}\big(\gamma_{\varepsilon}(\tau)\big).
$$
We have the following estimate for $\mathcal{D}_{\varepsilon}$.
\begin{proposition}\label{Mpr2.7}
(i) $\lim\limits_{\varepsilon\to 0}\frac{\mathcal{D}_{\varepsilon}}{\varepsilon^N}=\sum_{i=1}^kc_{\lambda_i}$.

(ii) $\limsup\limits_{\varepsilon\to 0}\frac{\max_{\tau\in\partial[0,T]^k}J_{\varepsilon}\big(\gamma_{\varepsilon}(\tau)\big)}{\varepsilon^N}\le\sum_{i=1}^kc_{\lambda_i}-\min\limits_{1\le i\le k}c_{\lambda_i}$.

(iii) For each $\delta>0$, there exists $\alpha>0$ such that for sufficiently small $\varepsilon>0$,
$$
\frac{J_{\varepsilon}(\gamma_{\varepsilon}(\tau))}{\varepsilon^N}\ge \frac{\mathcal{D}_{\varepsilon}}{\varepsilon^N}- \alpha
$$
implies that $\gamma_{\varepsilon}(\tau)\in \mathcal{X}^{\frac{\delta\varepsilon^{N/2}}{2}}_{\varepsilon}$.
\end{proposition}

\begin{proof}
By the decay rates of $U_{\lambda_i}$ and the analysis of \eqref{MMMMeq2.5}, we have
\begin{align*}
J_{\varepsilon}(\gamma_{\varepsilon}(\tau))/\varepsilon^N
&=\sum_{i=1}^kL_{\lambda_i}(t_iU_{\lambda_i}) + o_{\varepsilon}(1)\\
&\quad+\sum_{1\le i\ne j\le k}\frac{t_it_j}{2}\int_{\R^N\times\R^N}|x-y|^{-N-2s}\big(\eta_{\varepsilon}(x)
U_{\lambda_i}(x) - \eta_{\varepsilon}(y)U_{\lambda_i}(y)\big)\\
&\qquad\Big(\eta(\varepsilon x + p_i-p_j)U_{\lambda_j}\big(x+\frac{p_i-p_j}{\varepsilon}\big)-\eta(\varepsilon y + p_i-p_j)U_{\lambda_j}\big(y+\frac{p_i-p_j}{\varepsilon}\big)\Big)dxdy\\
&\quad+\sum_{i=1}^k\frac{t^2_i}{2}\int_{\R^N}\big(\eta^2_{\varepsilon}(x)V(\varepsilon x + p_i)-\lambda_i\big)U^2_{\lambda_i}(x)dx\\
&\quad+\sum_{i=1}^k\int_{\R^N}\Big(F\big(t_iU_{\lambda_i}(x)) - F\big(t_i\eta_{\varepsilon}(x)U_{\lambda_i}(x)\big)\Big),
%&= \sum_{i=1}^kL_{\lambda_i}(t_iU_{\lambda_i}) + \sum_{i=1}^kt^2_io_{\varepsilon}(1) + \sum_{i=1}^k\frac{t^p_i}{p}o_{\varepsilon}(1)\\
%&\quad + \sum_{i\ne j}^k\frac{t_it_j}{2}\int_{\R^N\times\R^N}\big(\eta_{\varepsilon}(x)
%U_{\lambda_i}(x) - \eta_{\varepsilon}(y)U_{\lambda_i}(y)\big)\Big(\eta(\varepsilon x + p_i-p_j)U_{\lambda_j}\big(x+\frac{p_i-p_j}{\varepsilon}\big)\\
%&\qquad\qquad-\eta(\varepsilon y + p_i-p_j)U_{\lambda_j}\big(y+\frac{p_i-p_j}{\varepsilon}\big)\Big){|x-y|^{-N-2s}}dxdy\\
%&\quad+\sum_{i=1}^k\frac{t^2_i}{2}\int_{\R^N}\big(\eta^2_{\varepsilon}(x)V(\varepsilon x + p_i)-\lambda_i\big)U^2_{\lambda_i}(x)dx\\
%&:=\sum_{i=1}^kL_{\lambda_i}(t_iU_{\lambda_i}) + o_{\varepsilon}(1)+\sum_{i\ne k}^k\frac{t_it_j}{2}I^{ij}_{\varepsilon},
%\big(\sum_{i=1}^kt_iW^i_{\varepsilon})
%&=  \sum_{{i = 1,k=1}\atop{i\ne j}}^{k}\frac{\varepsilon^{2s}}{2}\int_{\R^N\times\R^N}\frac{t_i(W^i_{\varepsilon}(x) - W^i_{\varepsilon}(y))(t_jW^j_{\varepsilon}(x) - W^j_{\varepsilon}(y))}{|x - y|^{N + 2s}}dxdy\\
%&\quad + \sum_{i = 1}^{k}\Big(\frac{t^2_i}{2}\|W^i_{\varepsilon}\|^2_{\varepsilon} - \frac{t^p_i}{p}\|W^i_{\varepsilon}\|^p_{L^p(\R^N)}\Big)\\
%&\le \sum_{i = 1}^{k}\Big(ct^2_i\|W^i_{\varepsilon}\|^2_{\varepsilon} - \frac{t^p_i}{p}\|W^i_{\varepsilon}\|^p_{L^p(\R^N)}\Big).
\end{align*}
where $\eta_{\varepsilon}(x)=\eta(\varepsilon x)$.
Choosing $\varepsilon>0$ be small enough such that $supp \eta_{\varepsilon}\cap supp\eta_{\varepsilon}\big(\cdot+ \frac{p_i-p_j}{\varepsilon}\big)=\emptyset$,  we have
\begin{align*}
&\Big|\int_{\R^N\times\R^N}|x-y|^{-N-2s}\big(\eta_{\varepsilon}(x)
U_{\lambda_i}(x) - \eta_{\varepsilon}(y)U_{\lambda_i}(y)\big)\\
&\qquad\Big(\eta(\varepsilon x + p_i-p_j)U_{\lambda_j}\big(x+\frac{p_i-p_j}{\varepsilon}\big)-\eta(\varepsilon y + p_i-p_j)U_{\lambda_j}\big(y+\frac{p_i-p_j}{\varepsilon}\big)\Big)dxdy\Big|\\
&=2\int_{B_{\frac{2\beta}{\varepsilon}}(0)}dx\int_{B_{\frac{2\beta}{\varepsilon}}\Big(\frac{p_j-p_i}{\varepsilon}\Big)}
\frac{\eta_{\varepsilon}(x)\eta_{\varepsilon}\big(y + \frac{p_i - p_j}{\varepsilon}\big)U_{\lambda_i}(x)U_{\lambda_j}\Big(y + \frac{p_i-p_j}{\varepsilon}\Big)}{|x - y|^{N+2s}}dy\\
&\le C\Big(\frac{\min\limits_{{i\ne j}\atop{1\le i,j\le k}}(|p_i-p_j|-4\beta)}{\varepsilon}\Big)^{-N-2s}\\
&=o_{\varepsilon}(1).
\end{align*}
%Now we estimate $I^{ijl}_{\varepsilon},1\le l\le 4$.
%Since $|\eta_{\varepsilon}(y)-\eta_{\varepsilon}(x)|\le \varepsilon\sup\limits_{x\in\R^N}|\nabla \eta(x)||y-x|$, there hold
%\begin{align*}
%|I^{ij1}_{\varepsilon}|
%&\le C\varepsilon^2\int_{\R^N}U_{\lambda_i}(x)U_{\lambda_j}\big(x + \frac{p_i-p_j}{\varepsilon}\big)dx\int_{B_1(x)}\frac{1}{|x - y|^{N+2s-2}}dy\\
%&\quad + C\int_{\R^N}U_{\lambda_i}(x)U_{\lambda_j}\big(x + \frac{p_i-p_j}{\varepsilon}\big)dx\int_{B^c_1(x)}\frac{1}{|x - y|^{N+2s}}dy\\
%&=o_{\varepsilon}(1).
%\end{align*}
%Similarly, it holds $|I^{ij4}_{\varepsilon}|\le o_{\varepsilon}(1)$. By the same estimates in Appendix, we conclude that $|I^{ij2}_{\varepsilon}|\le o_{\varepsilon}(1)$.
Then by the fact that $p_i\in\mathcal{M}_i$ and $t_i\le T$, $1\le i\le k$, we have
\begin{align}\label{Meq2.4}
\frac{J_{\varepsilon}(\gamma_{\varepsilon}(\tau))}{\varepsilon^N}
&=\sum_{i=1}^kL_{\lambda_i}(t_iU_{\lambda_i})+o_{\varepsilon}(1).
\end{align}
Hence we get (i) and obviously (ii) is true.

Finally, %let $\tau_{\varepsilon}$ satisfy $\mathcal{D}_{\varepsilon}= J_{\varepsilon}(\gamma_{\varepsilon}(\tau_{\varepsilon}))=\max_{\tau\in[0,T]^k}J_{\varepsilon}
%(\gamma_{\varepsilon}(\tau))$.
 \eqref{Meq2.4} implies that if $\tau_{\varepsilon}\in[0,T]^k$ satisfies $\lim\limits_{\varepsilon\to 0}\Big(\frac{J_{\varepsilon}\big(\gamma_{\varepsilon}(\tau_{\varepsilon})\big)}{\varepsilon^N} - \frac{\mathcal{D}_{\varepsilon}}{\varepsilon^N}\Big)=0$, then it must hold
$$
\lim_{\varepsilon\to 0}\tau_{\varepsilon}= (1,\ldots,1),
$$
which implies  (iii).

Consequently,  we complete the proof.
\end{proof}

Next, we define
$$
\mathcal{C}_{\varepsilon}=\inf_{\psi\in\Psi_{\varepsilon}}\max_{\tau\in[0,T]^k}
J_{\varepsilon}(\psi(\tau)),
$$
where
\begin{align}\label{Meq2.5}
\Psi_{\varepsilon}:=\big\{\psi_{\varepsilon}\in C\big(([0,T]^k,\mathcal{D}^s_{V,\varepsilon}(\mathbb{R}^N)\cap \mathcal{X}^{\nu\varepsilon^{N/2}}_{\varepsilon}\big)&|\psi_{\varepsilon}(\tau)=\gamma_{\varepsilon}(\tau)\ \text{for}\ \tau\in\partial[0,T]^k\big\},
\end{align}
where $\nu>0$ is large positive constant.
Obviously, $\Psi_{\varepsilon}$ is nonempty since $\gamma_{\varepsilon}\in \Psi_{\varepsilon}$. We now prove the following property of $\mathcal{C}_{\varepsilon}$.

\begin{lemma}\label{le2.9}
$$
\lim\limits_{\varepsilon\to 0} \frac{\mathcal{C}_{\varepsilon}}{\varepsilon^N} = \sum_{j = 1}^kc_{\lambda_j}.
$$
\end{lemma}
The proof  will rely on the following lemma, whose proof, for the sake of continuity, is postponed to the appendix. We define for every $i\in\{1,\ldots,k\}$, the functional $J^i_{\varepsilon}:W^{s,2}(S_i)\to \mathbb{R}$ as
$$
J^i_{\varepsilon}(u) = \frac{\varepsilon^{2s}}{2}\int_{S_i}\int_{S_i}\frac{|u(x) - u(y)|^2}{|x - y|^{N + 2s}}dy + \frac{1}{2}\int_{S_i}V(x)|u|^2
- \int_{S_i}\mathfrak{G}_{\varepsilon}(u) .
$$
We have
\begin{lemma}\label{le2.8}
The mountain pass value
$$
c^i_{\varepsilon} := \inf_{\gamma^i_{\varepsilon}\in\Gamma^i_{\varepsilon}}\max_{t\in [0,1]}J^i_{\varepsilon}(\gamma^i_{\varepsilon}(t)),\,\,i\in\{1,\ldots,k\}
$$
can be achieved, where
\begin{equation*}
\Gamma^i_{\varepsilon} : = \{\gamma^i_{\varepsilon}\in (C[0,1],W^{s,2}(S_i)):\gamma^i_{\varepsilon}(0) = 0,\ J^i_{\varepsilon}(\gamma^i_{\varepsilon}(1)) < 0\}.
\end{equation*}
Moreover,
\begin{equation}\label{eq2.3}
\lim_{\varepsilon\to 0}\frac{c^i_{\varepsilon}}{\varepsilon^N} = c_{\lambda_i}.
\end{equation}
\end{lemma}

Now we prove Lemma \ref{le2.9}:
\begin{proof}[\textbf{Proof of Lemma \ref{le2.9}}]
By Proposition \eqref{Mpr2.7}, we have the upper bounds
\begin{equation*}
   \limsup_{\varepsilon\to 0}\frac{\mathcal{C}_{\varepsilon}}{\varepsilon^N}\leq \sum_{j = 1}^kc_{\lambda_j}.
\end{equation*}

It remains to prove the lower estimate, i.e.,
\begin{equation*}\label{eq2.13}
  \liminf_{\varepsilon\to 0}\frac{\mathcal{C}_{\varepsilon}}{\varepsilon^N}\geq \sum_{j = 1}^kc_{\lambda_j}.
\end{equation*}
We first observe that given any $\psi_{\varepsilon}\in \Psi_{\varepsilon}$ and any continuous curve $c:[0,1]\to [0,T]^k$ with $c(0)\in \{0\}\times[0,T]^{k - 1}$ and $c(1)\in \{T\}\times [0,T]^{k - 1}$, we have $\gamma^1_{\varepsilon} = \psi_{\varepsilon} \circ c|_{S_1}\in \Gamma^1_{\varepsilon}$. In fact, by the definition of $\Psi_{\varepsilon}$, we have
$$
\gamma^1_{\varepsilon}(0) = 0,\ J^1_{\varepsilon}(\gamma^1_{\varepsilon}(1)) \le J_{\varepsilon}(TW^1_{\varepsilon}+0\cdot\sum_{i=2}^kW^i_{\varepsilon})  < 0.
$$
Lemma \ref{le2.8} implies that
$$
\sup_{t\in[0,1]}J^1_{\varepsilon}(\gamma^1_{\varepsilon}(t)) \geq \varepsilon^N (c_{\lambda_1} + o_{\varepsilon}(1)).
$$
Similarly, for every $\gamma^j_{\varepsilon} = \gamma \circ c|_{S_j} $ belongs to $\Gamma^j_{\varepsilon}$, where $c$ is  arbitrary continuous path which joint $[0,T]^{j - 1}\times\{0\}\times[0,T]^{k - j}$ with $[0,T]^{j - 1}\times\{T\}\times[0,T]^{k - j}$, it holds
$$
\sup_{t\in[0,1]}J^j_{\varepsilon}(\gamma^j_{\varepsilon}(t)) \geq \varepsilon^N (c_{\lambda_j} + o_{\varepsilon}(1)).
$$
Thus we can repeat the argument of Coti-Zetali and Rabinowitz in  \cite{CP} to prove, for every path $\psi_{\varepsilon}\in \Gamma$, the existence of a point $\hat{\tau}\in [0,1]^k$ satisfying
$$
J^i_{\varepsilon}(\psi_{\varepsilon}(\hat{\tau}))\geq \varepsilon^N(c_{\lambda_i} + o_{\varepsilon}(1))\ \text{for}\ j = 1,\ldots,k.
$$
Consequently, by \eqref{eq2.1}, \eqref{MMMeq2.2} and the fact that $\psi_{\varepsilon}(\tau)\in \mathcal{X}^{\nu\varepsilon^{N/2}}_{\varepsilon}$, we get
\begin{eqnarray}\label{Beq2.10}
\begin{split}
&\quad\liminf_{\varepsilon\to 0}\frac{1}{\varepsilon^{N}}\sup_{\tau\in[0,1]^k}J_{\varepsilon}(\psi_{\varepsilon}(\tau))\\
&\ge \liminf_{\varepsilon\to 0}\frac{1}{\varepsilon^{N}}J_{\varepsilon}(\psi_{\varepsilon}(\hat{\tau}))\\
& \geq \liminf_{\varepsilon\to 0}\frac{1}{\varepsilon^N}
\Big(\sum_{i = 1}^kJ^i_{\varepsilon}(\psi_{\varepsilon}(\hat{\tau})) - \varepsilon^{\kappa+2s}\int_{\mathbb{R}^N}|(-\Delta)^{s/2}\psi_{\varepsilon}({\hat{\tau}})|^2dx\Big)\\
&\geq \sum_{i =1}^kc_{\lambda_i},
\end{split}
\end{eqnarray}
which is exactly the required lower estimate.
\end{proof}

Next, we are going to construct a penalized solution for the penalized problem \eqref{eq2.2}. We first prove that the limit of a $(PS)_c$ sequence near the  set $\mathcal{X}_{\varepsilon}$ must own $k$-peaks.
\begin{proposition}\label{Mpr2.10}
Let $\{\varepsilon_j\}_j$  with $\lim\limits_{j\to\infty}\varepsilon_j = 0$ and $\{u_{\varepsilon_j}\}\subset \mathcal{X}^{d\varepsilon^{N/2}_j}_{\varepsilon_j}$ satisfy
$$
\lim_{j\to\infty}\frac{J_{\varepsilon_j}(u_{\varepsilon_j})}{\varepsilon^N_j}\le \sum_{i=1}^kc_{\lambda_i},\,\,\lim_{j\to\infty}\frac{\|J'_{\varepsilon_j}(u_{\varepsilon_j})\|}{\varepsilon^{N/2}_j}=0.
$$
Then for sufficiently small $d>0$, there exist, up to subsequence, $\{x^i_j\}_j\subset\R^3$, $i= 1,\ldots,k$, $x_i\in \mathcal{M}_i$, $\overline{U}_{\lambda_i}\in S_{\lambda_i}$ such that
\begin{equation}\label{Meq2.8}
\lim_{j\to\infty}{x^i_j}=x_i
\end{equation}
and
\begin{equation}\label{Meq2.9}
\lim_{j\to\infty}\Big\|u_{\varepsilon_j}(\cdot) - \sum_{i = 1}^k\eta(\cdot-x^i_j)\overline{U}_{\lambda_i}\Big(\frac{\cdot-x^i_j}{\varepsilon_j}\Big)\Big\|_{{\mathcal{D}}^s_{V,\varepsilon_j}}/\varepsilon^{N/2}_j = 0.
\end{equation}

\end{proposition}

\begin{proof}

For the sake of convenience, we write $\varepsilon$ for $\varepsilon_j$. Since $S_{\lambda_i}$, $i=1,\ldots,k$ are compact in $H^s(\R^N)$, there exist ${U}_{\lambda_i}\in S_{\lambda_i}$ and $p^i_{\varepsilon}\in \mathcal{M}_i$ such that
$$
\Big\|u_{\varepsilon}(x) - \sum_{i=1}^k\eta(x-p^i_{\varepsilon})U_{\lambda_i}\Big(\frac{x-p^i_{\varepsilon}}{\varepsilon}\Big)\Big\|_{\mathcal{D}^s_{V,\varepsilon}}
\le 2d\varepsilon^{N/2}.
$$
Letting $R_0\ge 1$ be a fixed positive constant and $\varepsilon R_0\le \beta$, for each $i=1,\ldots,k$, we have
%$$
%\inf_{p_i\in\mathcal{M}_i}\int_{B_{R_0}}|u_{\varepsilon}(\varepsilon x+p_i)-U_{\lambda_i}(x)|^2\le \frac{d^2}{\lambda_i},
%$$
%from which we deduce that there exists a $p^i_{\varepsilon}\in\mathcal{M}_i$ such that
$$
\int_{B_{R_0}}|u_{\varepsilon}(\varepsilon x+p^i_{\varepsilon})-U_{\lambda_i}(x)|^2\le \frac{4d^2}{\lambda_i}.
$$
As a result, we can let  $d>0$ be small enough so that
\begin{equation}\label{Meq2.8}
\liminf\limits_{\varepsilon\to0}\int_{B_{R_0}}|u_{\varepsilon}(\varepsilon x+p^i_{\varepsilon})|^2>0\ \text{and}\ \liminf_{\varepsilon\to 0}\|u_{\varepsilon}\|_{L^{\infty}(B_{\varepsilon R_0}(p^i_{\varepsilon}))}>0,
\end{equation}
for all $1\le i\le k$.
%In the following, for any function $f:\R^N\to\R$, we define $f^i_{\varepsilon}(\cdot) :=f_{\varepsilon}(\varepsilon\cdot+p^i_{\varepsilon})$.

Denote $u^{1,i}_{\varepsilon}(x)=\eta(x-p^i_{\varepsilon})u_{\varepsilon}(x)$, $u^1_{\varepsilon}(x)=\sum_{i=1}^ku^{1,i}_{\varepsilon}(x)$ and $u^2_{\varepsilon}(x)=u_{\varepsilon}(x)-u^1_{\varepsilon}(x)$. Denote $v^{1,i}_{\varepsilon}(x) = u^{1,i}_{\varepsilon}(\varepsilon x + p^i_{\varepsilon})$ and $v^{2,i}_{\varepsilon}(x) = v^i_{\varepsilon}(x) - v^{1,i}_{\varepsilon}(x)$, where $v^i_{\varepsilon}(x) = u_{\varepsilon}(\varepsilon x + p^i_{\varepsilon})$.
Fix arbitrarily an  $i\in\{1,\ldots,k\}$. Obviously, by assumption, for each $\varphi\in C^{\infty}_c(\R^N)$ and $\varepsilon$ small enough, testing $J'_{\varepsilon}(u_{\varepsilon})$ with $\varphi\Big(\frac{x-p^i_{\varepsilon}}{\varepsilon}\Big)$, we find
\begin{equation}
o_{\varepsilon}(1)=\int_{\R^N}\big((-\Delta)^sv^{1,i}_{\varepsilon}\big)\varphi + V^i_{\varepsilon}(x)v^{1,i}_{\varepsilon}\varphi-g_{\varepsilon}(\varepsilon x + p^i_{\varepsilon},v^{1,i}_{\varepsilon})\varphi+ \int_{\R^N}\big((-\Delta)^sv^{2,i}_{\varepsilon}\big)\varphi.
\end{equation}
Since $\{u_{\varepsilon}\}\subset \mathcal{X}^{d\varepsilon^{N/2}}_{\varepsilon}$, by fractional Hardy inequality \eqref{eq1.4}, we have
\begin{align}\label{MMMeq2.14}
\nonumber &\quad\Big|\int_{\R^N}\big((-\Delta)^sv^{2,i}_{\varepsilon}\big)\varphi\Big|\\
\nonumber&
= \Big|\int_{supp\varphi}dx\int_{B^c_{{\beta}/{\varepsilon}}(0)}\frac{\varphi(x)v^{2,i}_{\varepsilon}(y)}{|x - y|^{N + 2s}}dy\Big|\\
\nonumber&\le \int_{supp\varphi}(\varphi(x))^2dx\int_{B^c_{{\beta}/{\varepsilon}}(0)}\frac{1}{|x - y|^{N + 2s}}dy\\
\nonumber&\quad + \int_{supp\varphi}dx\int_{B^c_{{\beta}/{\varepsilon}}(0)}\frac{(v^{2,i}_{\varepsilon}(y))^2}{|x - y|^{N + 2s}}dy\\
\nonumber&= o_{\varepsilon}(1) + \int_{supp\varphi}dx\int_{B^c_{{\beta}/{\varepsilon}}(0)}\frac{(v^{2,i}_{\varepsilon}(y))^2}{|y|^{2s}}\frac{|y|^{2s}}{|x - y|^{N + 2s}}dy\\
&=o_{\varepsilon}(1).
\end{align}
Then, since $\{v^{1,i}_{\varepsilon}\}$ is bounded in $H^s(\R^N)$, by the Liouville type Theorem 3.3 of \cite{APX}, we have
\begin{equation}\label{MMMeq2.13}
(-\Delta)^sv^{1,i}_* + V(p^i_*)v^{1,i}_*=f\big((v^{1,i}_*)_+\big)\ \ \text{in}\ \R^N,
\end{equation}
where $v^{1,i}_*$ is the weak limit of some subsequence of $v^{1,i}_{\varepsilon}$ in $H^s(\R^N)$ and $p^i_*\in\mathcal{M}_i$ is limit of $p^i_{\varepsilon}$.
Consequently, according to  the argument of Proposition 3.4 in \cite{ADP-AA-2021}, we have for every $R>0$ that
\begin{align}\label{Meq2.14}
\nonumber
&\quad\liminf_{\varepsilon\to 0}\frac{J_{\varepsilon}(u_{\varepsilon})}{\varepsilon^N}\\
\nonumber&=o_{R}(1) + \liminf_{\varepsilon\to 0}\sum_{i=1}^k\int_{B_{\varepsilon R}(p^i_{\varepsilon})}\Big(\frac{1}{2}(|(-\Delta)^{s/2}u_{\varepsilon}|^2 + V(x)|u_{\varepsilon}(x)|^2) - \mathfrak{G}_{\varepsilon}(u_{\varepsilon})\Big)/\varepsilon^N\\
&\ge \sum_{i=1}^k L_{V(p^i_*)}(v^{1,i}_*)+o_R(1)\ge \sum_{i=1}^k c_{\lambda_i}+ o_R(1).
\end{align}
%Moreover, by the same argument of Lemma 3.4 in \cite{ADP-AA-2021}, we have
Consequently, by  Lemma \ref{le2.8}, we have $\lambda_i=V(p^i_*)$, $p^i_*\in\mathcal{M}_i$ and $v^{1,i}_*(\cdot+z_i)\in S_{\lambda_i}$ for some $z_i\in\R^N$.
Denote
$$
v^{1,i}_*(\cdot + z_i)= \overline{U}_{\lambda_i}.
$$

In the following we show that
\begin{equation}\label{MMeq2.15}
v^{1,i}_{\varepsilon}(\cdot)\to \overline{U}_{\lambda_i}(\cdot-z_i)\ \text{strongly in}\ H^s(\R^N).
\end{equation}
By the same argument of Lemma 3.4 in \cite{ADP-AA-2021}, we can conclude that
\begin{equation}\label{MMMeq2.17}
\lim_{{\varepsilon\to0}\atop{R\to\infty}}\|u_{\varepsilon}\|_{L^{\infty}\big(U\backslash\bigcup_{i=1}^k
B_{R\varepsilon}(p^i_{\varepsilon})\big)}=0
\end{equation}
and for any $r>0$, $y_{\varepsilon}\in\R^N$ with $\lim_{\varepsilon\to 0}\frac{|p^i_{\varepsilon}-y_{\varepsilon}|}{\varepsilon}=+\infty$, it holds
\begin{equation}\label{MMeq2.17}
\limsup_{{\varepsilon\to 0}}\int_{B_r(y_{\varepsilon})}|v^{1,i}_{\varepsilon}|^2=0.
\end{equation}
Then according to   Proposition \ref{pr2.1} and the Concentration-Compactness Lemma 1.21 of \cite{MW}, we have
\begin{equation}\label{MMeq2.18}
v^{1,i}_{\varepsilon}\to v^{1,i}_*\ \text{strongly in}\ L^{q}(\R^N),\ 2<q<2^*_s-1.
\end{equation}

By decomposition, one find
\begin{align*}
J_{\varepsilon}(u_{\varepsilon})
&=J_{\varepsilon}(u^1_{\varepsilon}) + \frac{\varepsilon^{2s}}{2}\int_{\R^N}|(-\Delta)^{s/2}u^2_{\varepsilon}|^2\\
&\quad+ \varepsilon^{2s}\int_{\R^N}dx\int_{R^N}\frac{\big(u^1_{\varepsilon}(x)-u^1_{\varepsilon}(y)\big)\big(u^2_{\varepsilon}(x)-u^2_{\varepsilon}(y)\big)}{|x-y|^{N+2s}}dy\\
&+\frac{1}{2}\int_{\R^N}V(x)|u^2_{\varepsilon}|^2 + \int_{\R^N}V(x)u^1_{\varepsilon}u^2_{\varepsilon} + \int_{\R^N}G_{\varepsilon}(u^1_{\varepsilon}) -\int_{\R^N}\mathcal{G}_{\varepsilon}(u_{\varepsilon}).
\end{align*}
But, with \eqref{MMMeq2.17} at hand, we can use the same method in the proof of \eqref{MMeq2.22}(which needs only \eqref{Meq2.22}) to show that
\begin{equation}\label{addeq2.22}
\frac{\varepsilon^{2s}}{2}\int_{\R^N}|(-\Delta)^{s/2}u^2_{\varepsilon}|^2 + \frac{1}{2}\int_{\R^N}V(x)|u^2_{\varepsilon}|^2 = \varepsilon^No_{\varepsilon}(1),
\end{equation}
which and \eqref{MMMeq2.2} imply that
\begin{align*}
J_{\varepsilon}(u_{\varepsilon})
&= J_{\varepsilon}(u^1_{\varepsilon})  + \int_{\R^N}F(u^1_{\varepsilon}) -\int_{\Lambda}F(u_{\varepsilon}) + \varepsilon^No_{\varepsilon}(1).
\end{align*}
From \eqref{MMMeq2.17}, we have
$$
\Big|\int_{\R^N}F(u^1_{\varepsilon}) -\int_{\Lambda}F(u_{\varepsilon})\Big|\le \|u_{\varepsilon}\|^{\tilde{\kappa}}_{L^{\infty}\Big(\Lambda\backslash \bigcup_{i=1}^kB_{\beta}(p^i_{\varepsilon})\Big)}\int_{\Lambda\backslash \bigcup_{i=1}^kB_{\beta}(p^i_{\varepsilon})}|u_{\varepsilon}|^2=\varepsilon^No_{\varepsilon}(1).
$$
Hence, by the analysis above, we have
\begin{align*}
J_{\varepsilon}(u_{\varepsilon})
&= J_{\varepsilon}(u^1_{\varepsilon})+\varepsilon^No_{\varepsilon}(1).
\end{align*}

Decomposing again, we find
\begin{align*}
 \frac{J_{\varepsilon}(u_{\varepsilon})}{\varepsilon^N} = \sum_{i=1}^kJ_{\varepsilon}(v^{1,i}_{\varepsilon}) +  \varepsilon^{-N} T^1_{\varepsilon}(\tilde{\eta}_{\varepsilon}) + o_{\varepsilon}(1),
\end{align*}
where
\begin{align*}
T^1_{\varepsilon}(\tilde{\eta}_{\varepsilon}):
= \varepsilon^{2s}\sum_{{1\le i\ne j\le k}}\int_{\R^N}\frac{\big(u^{1,i}_{\varepsilon}(x) - u^{1,i}_{\varepsilon}(y)\big)\big(u^{1,j}_{\varepsilon}(x) - u^{1,j}_{\varepsilon}(y)\big)}{|x-y|^{N+2s}}dy.
\end{align*}
But, it has been  proved in Appendix that
\begin{align}\label{MMMeq2.26}
T^1_{\varepsilon}(\tilde{\eta}_{\varepsilon}):
= \varepsilon^No_{\varepsilon}(1).
\end{align}
Hence, it holds
\begin{align*}
\frac{J_{\varepsilon}(u_{\varepsilon})}{\varepsilon^N}&=\sum_{i=1}^kJ_{\varepsilon}(v^{1,i}_{\varepsilon}) +o_{\varepsilon}(1).
\end{align*}
So
$$
\lim_{\varepsilon\to 0}\sum_{i=1}^kJ_{\varepsilon}(v^{1,i}_{\varepsilon})=\sum_{i=1}^kc_{\lambda_i},
$$
which combining with the analysis of \eqref{Meq2.14} yields
$$
\lim_{\varepsilon\to 0}J_{\varepsilon}(v^{1,i}_{\varepsilon})=c_{\lambda_i},\ i=1,\ldots,k.
$$
Consequently, by \eqref{MMeq2.18}, we have
\begin{align*}
\nonumber&\quad\int_{\R^N}|(-\Delta)^{s/2}U_{\lambda_i}(\cdot-z_i)|^2 + \lambda_i|U_{\lambda_i}(\cdot-z_i)|^2\\
&\ge\limsup_{\varepsilon\to 0}\int_{\R^N}|(-\Delta)^{s/2}v^{1,i}_{\varepsilon}|^2 + V^i_{\varepsilon}(x)|v^{1,i}_{\varepsilon}|^2\\
\nonumber&\ge \limsup_{\varepsilon\to 0}\int_{\R^N}|(-\Delta)^{s/2}v^{1,i}_{\varepsilon}|^2 + \lambda_i|v^{1,i}_{\varepsilon}|^2\\
\nonumber&\ge \int_{\R^N}|(-\Delta)^{s/2}U_{\lambda_i}(\cdot-z_i)|^2 + \lambda_i|U_{\lambda_i}(\cdot-z_i)|^2,
\end{align*}
which gives \eqref{MMeq2.15}.

Now from \eqref{MMeq2.15}, we have
%By Lemma 3.4 in \cite{ADP-AA-2021}, we have
%\begin{equation}
%\lim_{{\varepsilon\to0}\atop{R\to\infty}}\|u_{\varepsilon}\|_{L^{\infty}\big(U\backslash\bigcup_{i=1}^k
%B_{R\varepsilon}(p^i_{\varepsilon})\big)}=0.
%\end{equation}
%which and the choice of $\mathcal{P}_{\varepsilon}$ imply that
%$$
%\lim_{\varepsilon\to 0}\|u^{2}_{\varepsilon}\|_{\mathcal{D}^{s}_{V,\varepsilon}(\R^N)}/\varepsilon^{N/2}=0.
%$$
%Finally, we have
\begin{align*}
&\quad\varepsilon^{-N}\Big\|u_{\varepsilon} - \sum_{i=1}^k\eta(x-p^i_{\varepsilon}-\varepsilon z_i)U_{\lambda_i}\Big(\frac{x-p^i_{\varepsilon}-\varepsilon z_i}{\varepsilon}\Big)\Big\|^2_{\mathcal{D}^s_{V,\varepsilon}}\\
&\le 2\varepsilon^{-N}\Big\|\sum_{i=1}^k\eta(x-p^i_{\varepsilon}-\varepsilon z_i)\Big(u_{\varepsilon} - U_{\lambda_i}\Big(\frac{x-p^i_{\varepsilon}-\varepsilon z_i}{\varepsilon}\Big)\Big)\Big\|^2_{\mathcal{D}^s_{V,\varepsilon}}\\
&\quad+2\varepsilon^{-N}\Big\|u_{\varepsilon}-\sum_{i=1}^k\eta(x-p^i_{\varepsilon}-\varepsilon z_i)u_{\varepsilon}\Big\|^2_{\mathcal{D}^s_{V,\varepsilon}}\\
&\le 2k\varepsilon^{-N}\sum_{i=1}^k\Big\|\eta(x-p^i_{\varepsilon}-\varepsilon z_i)\Big(u_{\varepsilon} - U_{\lambda_i}\Big(\frac{x-p^i_{\varepsilon}-\varepsilon z_i}{\varepsilon}\Big)\Big)\Big\|^2_{\mathcal{D}^s_{V,\varepsilon}}\\
&\quad+2\varepsilon^{-N}\Big\|u_{\varepsilon}-\sum_{i=1}^k\eta(x-p^i_{\varepsilon}-\varepsilon z_i)u_{\varepsilon}\Big\|^2_{\mathcal{D}^s_{V,\varepsilon}}\\
&:=o_{\varepsilon}(1)+I_{\varepsilon}.\\
\end{align*}
It remains to show that
\begin{equation}\label{MMeq2.22}
I_{\varepsilon}=o_{\varepsilon}(1).
\end{equation}
By the same   blow-up analysis of lemmas 3.3 and 3.4 in \cite{ADP-AA-2021}, it holds
\begin{equation}\label{Meq2.22}
\lim_{{\varepsilon\to0}\atop{R\to\infty}}\|u_{\varepsilon}\|_{L^{\infty}\big(U\backslash\bigcup_{i=1}^k
B_{R\varepsilon}(p^i_{\varepsilon} + \varepsilon z_i)\big)}=0.
\end{equation}
Consequently, denoting $\tilde{\eta}_{\varepsilon} = 1 - \sum_{i=1}^k\eta\big(2({x-p^i_{\varepsilon}-\varepsilon z_i})\big)$ and testing $J'_{\varepsilon}(u_{\varepsilon})$ against with $\bar{\eta}_{\varepsilon}u_{\varepsilon}$, we have, for $\varepsilon>0$ small enough,
\begin{align*}
\tilde{I}_{\varepsilon}:&=\quad\frac{\varepsilon^{2s}}{2}\int_{\R^N}\tilde{\eta}_{\varepsilon}(x)dx\int_{\R^N}\frac{|u_{\varepsilon}(x) - u_{\varepsilon}(y)|^2}{|x - y|^{N + 2s}}dy + \int_{\R^N}V(x)|\tilde{\eta}_{\varepsilon}(x)||u_{\varepsilon}|^2dx\\
& \le   \int_{\R^N}\mathfrak{g}_{\varepsilon}(u_{\varepsilon})\tilde{\eta}_{\varepsilon}u_{\varepsilon} + \frac{\varepsilon^{2s}}{2}\int_{\R^N}dx\int_{\R^N}\frac{(u_{\varepsilon}(x) - u_{\varepsilon}(y))(\tilde{\eta}_{\varepsilon}(y)-\tilde{\eta}_{\varepsilon}(x))u_{\varepsilon}(y)}{|x - y|^{N + 2s}}dy\\
&\quad+o_{\varepsilon}(1)\varepsilon^{N/2}\|\tilde{\eta}_{\varepsilon}u_{\varepsilon}\|_{\mathcal{D}^s_{V,\varepsilon}}\\
&:= \int_{\R^N}\mathfrak{g}_{\varepsilon}(u_{\varepsilon})\tilde{\eta}_{\varepsilon}u_{\varepsilon} + T^2_{\varepsilon}(\tilde{\eta})+o_{\varepsilon}(1)\varepsilon^{N/2}\|\tilde{\eta}_{\varepsilon}u_{\varepsilon}\|_{\mathcal{D}^s_{V,\varepsilon}}\\
&\le \|u_{\varepsilon}\|^{\tilde{\kappa}}_{L^{\infty}\big(\Lambda\backslash \bigcup_{i=1}^kB_{R\varepsilon}(p^i_{\varepsilon} + \varepsilon z_i)\big)}\int_{\R^N}V(x)\tilde{\eta}_{\varepsilon}(x)|u_{\varepsilon}|^2dx + \int_{\R^N\backslash\Lambda}\mathcal{P}_{\varepsilon}|u_{\varepsilon}|^2\\
&\quad+ T^2_{\varepsilon}(\tilde{\eta}_{\varepsilon})+o_{\varepsilon}(1)\varepsilon^{N/2}\|\tilde{\eta}_{\varepsilon}u_{\varepsilon}\|_{\mathcal{D}^s_{V,\varepsilon}},
\end{align*}
which implies
\begin{align*}
&\tilde{I}_{\varepsilon}\le C\big(\int_{\R^N\backslash\Lambda}\mathcal{P}_{\varepsilon}|u_{\varepsilon}|^2 + T^2_{\varepsilon}(\tilde{\eta}_{\varepsilon})\big) +o_{\varepsilon}(1)\varepsilon^{N/2}\|\tilde{\eta}_{\varepsilon}u_{\varepsilon}\|_{\mathcal{D}^s_{V,\varepsilon}}.
\end{align*}
However, we have proved  in the Appendix that
\begin{equation}\label{MMMeq2.28}
\limsup_{\varepsilon\to 0}\frac{T^2_{\varepsilon}(\tilde{\eta}_{\varepsilon})}{\varepsilon^N}\le 0
\end{equation}
and
\begin{equation}\label{addeq2.27}
\|\tilde{\eta}_{\varepsilon}u_{\varepsilon}\|_{\mathcal{D}^s_{V,\varepsilon}}\le C\varepsilon^{N/2}.
\end{equation}
Hence, by the choice of $\mathcal{P}_{\varepsilon}$ and fractional Hardy inequality \eqref{eq1.4}, it holds
\begin{align}\label{MMMMeq2.28}
&\lim_{\varepsilon\to 0}\frac{\tilde{I}_{\varepsilon}}{\varepsilon^N}=0.
\end{align}
Noting the following estimate  proved in the Appendix
\begin{equation}\label{MMMeq2.29}
T^3_{\varepsilon}(\breve{\eta}_{\varepsilon})= \varepsilon^{2s} \int_{\R^N}\int_{\R^N}\frac{|\breve{\eta}_{\varepsilon}(x)u_{\varepsilon}(x) - \breve{\eta}_{\varepsilon}(y)u_{\varepsilon}(y)|^2}{|x-y|^{N+2s}}dxdy=\varepsilon^No_{\varepsilon}(1),
\end{equation}
where $\breve{\eta}_{\varepsilon}(x)=1-\sum_{i=1}^k\eta(x-p^i_{\varepsilon}-\varepsilon z_i)$, we find
\begin{align*}
I_{\varepsilon}&
\le \frac{T^3_{\varepsilon}(\breve{\eta}_{\varepsilon})}{\varepsilon^N}+\frac{\tilde{I}_{\varepsilon}}{\varepsilon^N}
=o_{\varepsilon}(1),
\end{align*}
%Then
%\begin{align*}
%&\lim_{\varepsilon\to 0}I_{\varepsilon}=0,
%\end{align*}
which is exactly  \eqref{MMeq2.22}. Letting $x^i_{\varepsilon} = p^i_{\varepsilon} + \varepsilon z_i$, we get
$$
\lim_{\varepsilon\to 0}\varepsilon^{-N}\Big\|u_{\varepsilon} - \sum_{i=1}^k\eta(x-x^i_{\varepsilon})U_{\lambda_i}\Big(\frac{x-x^i_{\varepsilon}}{\varepsilon}\Big)\Big\|^2_{\mathcal{D}^s_{V,\varepsilon}}=0.
$$
Hence we  complete the proof.

\end{proof}

\begin{proposition}\label{Mpr2.11}
For $d>0$ sufficiently small, there exist constants $\sigma>0$ and $\varepsilon_0>0$, such that
$$
\|J'_{\varepsilon}(u)\|_{\mathcal{D}^s_{V,\varepsilon}(\R^N)}\ge \varepsilon^{N/2}\sigma\ \text{for}\ J^{\mathcal{D}_{\varepsilon}}_{\varepsilon}\cap(\mathcal{X}^{d\varepsilon^{N/2}}_{\varepsilon}\backslash \mathcal{X}^{d\varepsilon^{N/2}/2}_{\varepsilon})\ \text{and}\ \varepsilon\in(0,\varepsilon_0),
$$
where $J^{\mathcal{D}_{\varepsilon}}_{\varepsilon}=\{u\in \mathcal{D}^s_{V,\varepsilon}(\R^N):J_{\varepsilon}(u)\le \mathcal{D}_{\varepsilon}\}$.
\end{proposition}

\begin{proof}
To the contrary, suppose that for small $d_1>d_2>0$, there exist $\{\varepsilon_j\}^{\infty}_{j=1}$ with $\lim\limits_{j\to\infty}\varepsilon_j = 0$ and $u_{\varepsilon_j}\in \mathcal{X}^{d_1\varepsilon^{N/2}_j}_{\varepsilon_j}\backslash \mathcal{X}^{d_2\varepsilon^{N/2}_j}_{\varepsilon_j}$ satisfying $\lim\limits_{j\to\infty}J_{\varepsilon_j}(u_{\varepsilon_j})/\varepsilon^N_j\le \sum_{i=1}^kc_{\lambda_i}$ and $\lim\limits_{j\to\infty}\frac{J'_{\varepsilon_j}(u_{\varepsilon_j})}{\varepsilon^{N/2}_j}=0$. By Proposition \ref{Mpr2.10}, there exists $\{x^i_j\}^{\infty}_{j=1}\subset \R^N$, $i=1,\ldots,k$, $x_i\in \mathcal{M}_i$, such that
$$
\lim_{j\to\infty}|x^i_j - x_i|=0\ \text{and}\ \lim_{j\to\infty}\Big\|u_{\varepsilon_j}(\cdot) - \sum_{i = 1}^k\eta(\cdot-x^i_j)U_{\lambda_i}\Big(\frac{\cdot-x^i_j}{\varepsilon_j}\Big)\Big\|_{{\mathcal{D}}^s_{V,\varepsilon_j}}/\varepsilon^{N/2}_j = 0.
$$
Hence, by the definition of $\mathcal{X}_{\varepsilon}$, we see that $\lim\limits_{j\to\infty}\text{dist}(u_{\varepsilon_j},\mathcal{X}_{\varepsilon_j})/\varepsilon^{N/2}_j=0$. This is a contradiction to $u_{\varepsilon_j}\not\in \mathcal{X}^{d_2 \varepsilon^{N/2}_j/2}_{\varepsilon_j}$.
\end{proof}

\vspace{0.5cm}

Now, we use Proposition \ref{Mpr2.11} and the Deformation Lemma 2.2 in \cite{MW} to construct a $(PS)_c$ sequence near the set $\mathcal{X}_{\varepsilon}$.

Define
$$
\mu:=\varepsilon^{-N}\inf_{u\in\mathcal{X}_{\varepsilon}}
\{\|u\|_{\varepsilon,S_i},i=1,\ldots,k\}.
$$
Fix $d_0\in(0,\frac{\mu}{2})$ such that Propositions \ref{Mpr2.10} and \ref{Mpr2.11} hold for $d\in(0,d_0]$.

\begin{proposition}\label{Mpr3.3}
For sufficiently small fixed $\varepsilon>0$, there exists a sequence $\{u_n\}^{\infty}_{n=1}\subset J^{\mathcal{D}_{\varepsilon}}_{\varepsilon}\cap\mathcal{X}^{d\varepsilon^{N/2}}_{\varepsilon}$ such that $J'_{\varepsilon}(u_n)\to 0$ as $n\to\infty$.
\end{proposition}

\begin{proof}
By Proposition \ref{Mpr2.11}, there exists a constant $\sigma\in(0,1)$, such that
$$
\|J'_{\varepsilon}(u)\|_{\mathcal{D}^s_{V,\varepsilon}(\R^N)}\ge \varepsilon^{N/2}\sigma\ \text{for}\ u\in J^{\mathcal{D}_{\varepsilon}}_{\varepsilon}\cap(\mathcal{X}^{d\varepsilon^{N/2}}_{\varepsilon}\backslash \mathcal{X}^{d\varepsilon^{N/2}/2}_{\varepsilon})\ \text{and}\ \varepsilon\in(0,\varepsilon_0).
$$

From Proposition \ref{Mpr2.7}(iii), there exist constants $\alpha>0$, $\varepsilon_1(\alpha)>0$ such that for $\varepsilon\in(0,\varepsilon_1]$ and $d\in(0,d_0]$, that
\begin{equation}\label{MMMMeq2.27}
J_{\varepsilon}(\gamma_{\varepsilon}(\tau))/\varepsilon^N\ge \mathcal{D}_{\varepsilon}/\varepsilon^N-\alpha\Rightarrow \gamma_{\varepsilon}(\tau)\in \mathcal{X}^{\varepsilon^{N/2}d/2}_{\varepsilon}.
\end{equation}
Now, set
$$
\alpha_0:=\min \{\frac{\alpha}{2},\frac{1}{8}\sigma^2d_0,\frac{\rho}{2}\},
$$
where $\rho = \min\limits_{1\le i\le k}c_{\lambda_i}$. We choose $0<\bar{\varepsilon}<\min\{\varepsilon_0,\varepsilon_1\}$ such that for $\varepsilon\in(0,\bar{\varepsilon}]$
$$
|\mathcal{D}_{\varepsilon}/\varepsilon^N-\sum_{i=1}^kc_{\lambda_i}|<\alpha_0,\ |\mathcal{C}_{\varepsilon}/\varepsilon^N-\sum_{i=1}^kc_{\lambda_i}|<\alpha_0\ \text{and}\ |\mathcal{D}_{\varepsilon}/\varepsilon^N - \mathcal{C}_{\varepsilon}/\varepsilon^N|<\alpha_0.
$$

We assume to the contrary that for some $\varepsilon\in(0,\bar{\varepsilon}]$, $d\in(0,d_0)$, there exist $\beta=\beta(\varepsilon)\in(0,1)$ such that
$$
\|J'_{\varepsilon}(u)\|/\varepsilon^{N/2}\ge \beta>0\ \text{for}\ u\in J^{\mathcal{D}_{\varepsilon}}_{\varepsilon}\cap\mathcal{X}^{d\varepsilon^{N/2}}_{\varepsilon}.
$$

By Lemma 2.2 in \cite{MW}, we can choose $g_{\varepsilon}$ be a pseudo-gradient vector field for $J'_{\varepsilon}$ on a neighbourhood $N_{\varepsilon}$ of $J^{\mathcal{D}_{\varepsilon}}_{\varepsilon}\cap \mathcal{X}^{d\varepsilon^{N/2}}_{\varepsilon}$, which satisfies
\begin{align*}
\|g_{\varepsilon}(u)\|&\le 2\min\{\varepsilon^{N/2},\|J'_{\varepsilon}(u)\|\},\\
\langle J'_{\varepsilon}(u),g_{\varepsilon}(u)\rangle&\ge \min\{\varepsilon^{N/2},\|J'_{\varepsilon}(u)\|\}\|J'_{\varepsilon}(u)\|.
\end{align*}

Let $\zeta_{\varepsilon}$ be a Lipschitz continuous function on $\mathcal{D}^s_{V,\varepsilon}(\R^N)$ such that $0\le \zeta_{\varepsilon}\le 1,\ \zeta_{\varepsilon}\equiv 1$ on $\mathcal{X}^{d\varepsilon^{N/2}}_{\varepsilon}\cap J^{\mathcal{D}_{\varepsilon}}_{\varepsilon}$ and $\zeta_{\varepsilon}\equiv 0$ on $\mathcal{D}^s_{V,\varepsilon}(\R^N)\backslash N_{\varepsilon}$. Let $\xi_{\varepsilon}$ be a Lipschitz continuous function on $\R$ such that $0\le \xi_{\varepsilon}\le 1$, $\xi_{\varepsilon}(l)\equiv 1$ if $|l-\mathcal{D}_{\varepsilon}\varepsilon^{-N}|\le \frac{\alpha}{2}$ and $\xi_{\varepsilon}(l)\equiv 0$ if $|l-\mathcal{D}_{\varepsilon}\varepsilon^{-N}|\ge \alpha$. Set
\begin{equation}\label{Meq3.2}
h_{\varepsilon}(u):=\left\{
                      \begin{array}{ll}
                        -\zeta_{\varepsilon}(u)\xi_{\varepsilon}(\varepsilon^{-N}J_{\varepsilon}(u))g_{\varepsilon}(u), & \text{if}\ u\in N_{\varepsilon} \\
                        0, & \text{if}\ u\in \mathcal{D}^s_{V,\varepsilon}\backslash N_{\varepsilon}.
                      \end{array}
                    \right.
\end{equation}
Then there exists a unique solution $\Phi_{\varepsilon}:\mathcal{D}^s_{V,\varepsilon}\times[0,+\infty)\to \mathcal{D}^s_{V,\varepsilon}$ to the following initial value problem
\begin{equation}\label{Meq3.3}
\left\{
  \begin{array}{ll}
    \frac{d}{d\theta}\Phi_{\varepsilon}(u,\theta)=
h_{\varepsilon}(\Phi_{\varepsilon}(u,\theta)), &\\
    \Phi_{\varepsilon}(u,0)=u.&
  \end{array}
\right.
\end{equation}
(See the proof of Lemma 2.3 in \cite{MW}). It can be easily check that $\Phi_{\varepsilon}$ has the following properties:
\begin{align}\label{MMeq3.4}
&\nonumber(1)\ \Phi_{\varepsilon}(u,\theta)=u\ \text{if}\ \theta=0\ \text{or}\ u\in \mathcal{D}^s_{V,\varepsilon}(\R^N)\backslash N_{\varepsilon}\ \text{or}\ |J_{\varepsilon}(u) - \ \mathcal{D}_{\varepsilon}|\ge \alpha\varepsilon^{N}.\\
&(2)\|\frac{d}{d\theta}\Phi_{\varepsilon}(u,\theta)\|\le 2\varepsilon^{N/2}.\\
\nonumber&(3)\ \frac{d}{d\theta}J_{\varepsilon}(\Phi_{\varepsilon}(u,\theta)) = \langle J'_{\varepsilon}(\Phi_{\varepsilon}(u,\theta)),h_{\varepsilon}(\Phi_{\varepsilon}(u,\theta))\le 0.
\end{align}

\vspace{0.2cm}
\noindent\textbf{Claim 1} For any $\tau\in[0,T]^k$, there exists $\theta_{\tau}\in[0,+\infty)$ such that
$$
\Phi_{\varepsilon}(\gamma_{\varepsilon}(\tau),\theta_{\tau})\in J^{\mathcal{D}_{\varepsilon}-\alpha_0\varepsilon^N}_{\varepsilon}.
$$

\textit{Proof of Claim 1.} Assume by contradiction that there exists $\tau_0\in[0,T]^k$ such that
\begin{equation}\label{Meq3.4'}
J_{\varepsilon}(\Phi_{\varepsilon}(\gamma_{\varepsilon}(\tau_0),\theta))>\mathcal{D}_{\varepsilon} - \alpha_0\varepsilon^N
\end{equation}
for all $\theta>0$. Then, by the property (3) in \eqref{MMeq3.4}, we have
\begin{equation}\label{MMeq3.5}
\mathcal{D}_{\varepsilon}-\alpha_0\varepsilon^{N}< J_{\varepsilon}(\Phi_{\varepsilon}(\gamma_{\varepsilon}(\tau_0),\theta))\le J_{\varepsilon}(\Phi_{\varepsilon}(\gamma_{\varepsilon}(\tau_0),0))=J_{\varepsilon}(\gamma_{\varepsilon}(\tau_0))\le \mathcal{D}_{\varepsilon}<\mathcal{D}_{\varepsilon}+\alpha_0\varepsilon^{N},
\end{equation}
which and the choice of $\alpha_0$ imply that $\xi_{\varepsilon}(\varepsilon^{-N}J_{\varepsilon}(\Phi_{\varepsilon}(\gamma_{\varepsilon}(\tau_0),\theta)))\equiv 1$.

If $\Phi_{\varepsilon}(\gamma_{\varepsilon}(\tau_0),\theta)\in \mathcal{X}^{d\varepsilon^{N/2}}_{\varepsilon}$ for all $\theta\ge 0$, then by \eqref{MMeq3.5}, we have $\Phi_{\varepsilon}(\gamma_{\varepsilon}(\tau_0),\theta)\in \mathcal{X}^{d\varepsilon^{N/2}}_{\varepsilon}\cap J^{\mathcal{D}_{\varepsilon}}_{\varepsilon}$ for all $\theta \ge 0$. Then $\zeta_{\varepsilon}(\Phi_{\varepsilon}(\gamma_{\varepsilon}(\tau_0),\theta))\equiv 1$ and $|\frac{d}{d\theta}J_{\varepsilon}(\Phi_{\varepsilon}(\gamma_{\varepsilon}(\tau_0),\theta))|\ge\beta^2\varepsilon^N$ for all $\theta\ge 0$. Hence
\begin{align*}
 J_{\varepsilon}(\Phi_{\varepsilon}(\gamma_{\varepsilon}(\tau_0),\frac{\alpha}{\beta^2})
\le \mathcal{D}_{\varepsilon}+\alpha_0\varepsilon^N - \varepsilon^N \int_0^{\frac{\alpha}{\beta^2}}\beta^2d\theta\le \mathcal{D}_{\varepsilon}-\alpha_0\varepsilon^N,
\end{align*}
a contradiction to \eqref{MMeq3.5}.

Assume that $\Phi_{\varepsilon}(\gamma_{\varepsilon}(\tau_0),\theta_0)\not\in \mathcal{X}^{d\varepsilon^{N/2}}_{\varepsilon}$ for some $\theta_0>0$. Note that \eqref{Meq3.4'}, \eqref{MMeq3.5} and \eqref{MMMMeq2.27} imply that $\gamma_{\varepsilon}(\tau_0)
\in \mathcal{X}^{\frac{d}{2}\varepsilon^{N/2}}_{\varepsilon}$. Then there exist $0<\theta^1_0<\theta^2_0$ such that $\Phi_{\varepsilon}(\gamma_{\varepsilon}(\tau_0),\theta^1_0)\in \partial \mathcal{X}^{\frac{d}{2}\varepsilon^{N/2}}_{\varepsilon}$, $\Phi_{\varepsilon}(\gamma_{\varepsilon}(\tau_0),\theta^2_0)\in \partial \mathcal{X}^{d\varepsilon^{N/2}}_{\varepsilon}$ and $\Phi_{\varepsilon}(\gamma_{\varepsilon}(\tau_0),\theta)\in  \mathcal{X}^{d\varepsilon^{N/2}}_{\varepsilon}\backslash\mathcal{X}^{\frac{d}{2}\varepsilon^{N/2}}_{\varepsilon}$ for all $\theta\in(\theta^1_0,\theta^2_0)$. Then by Proposition \ref{Mpr2.11}, we have $|\frac{d}{d\theta}J_{\varepsilon}(\Phi_{\varepsilon}(\gamma_{\varepsilon}(\tau_0),\theta)|\ge \sigma^2\varepsilon^{N}$ for all $\theta\in(\theta^1_0,\theta^2_0)$. By property (2) of \eqref{MMeq3.4} and mean value theorem, we have
$$
\frac{d\varepsilon^{N/2}}{2}\le \|\Phi_{\varepsilon}(\gamma_{\varepsilon}(\tau_0),\theta^1_0)-\Phi_{\varepsilon}(\gamma_{\varepsilon}(\tau_0),\theta^2_0)\|\le 2\varepsilon^{N/2}|\theta^1_0-\theta^2_0|,
$$
which implies
$$
|\theta^1_0-\theta^2_0|\ge\frac{d}{4}.
$$
Hence
\begin{align}\label{Meq3.7}
J_{\varepsilon}\big(\Phi_{\varepsilon}(\gamma_{\varepsilon}(\tau_0),\theta^2_0)\big) \nonumber&=J_{\varepsilon}\big(\Phi_{\varepsilon}(\gamma_{\varepsilon}(\tau_0),\theta^1_0)\big) + \int_{\theta^1_0}^{\theta^2_0}\frac{d}{d\theta}J_{\varepsilon}\big(\Phi_{\varepsilon}(\gamma_{\varepsilon}(\tau_0),\theta)\big)d\theta\\
\nonumber&\le \mathcal{D}_{\varepsilon}+ \alpha_0\varepsilon^N - \varepsilon^N \sigma^2|\theta^1_0-\theta^2_0|\\
\nonumber&<\mathcal{D}_{\varepsilon}+ \alpha_0\varepsilon^N - \varepsilon^N \sigma^2\frac{d}{4}\\
\nonumber&\le \mathcal{D}_{\varepsilon}+ \alpha_0\varepsilon^N - \varepsilon^N \sigma^2\frac{d_0}{4}\\
&\le \mathcal{D}_{\varepsilon} - \alpha_0\varepsilon^N,
\end{align}
which is a contradiction to \eqref{MMeq3.5}. This completes the proof of Claim 1.

By Claim 1, we can define $\theta(\tau):=\inf\{\theta\ge 0:J_{\varepsilon}\big(\Phi_{\varepsilon}(\gamma_{\varepsilon}(\tau),\theta)\big)\le \mathcal{D}_{\varepsilon} - \alpha_0\varepsilon^N\}$ and let $\bar{\gamma}_{\varepsilon}(\tau):= \Phi_{\varepsilon}(\gamma_{\varepsilon}(\tau),\theta(\tau))$. We have

\vspace{0.2cm}

\noindent\textbf{Claim 2} $\bar{\gamma}_{\varepsilon}(\tau)\in\Psi_{\varepsilon}$.

\vspace{0.15cm}
\textit{Proof of Claim 2.} Firstly, for any $\tau\in\partial[0,T]^k$, by Proposition \ref{Mpr2.7}, we have $\gamma_{\varepsilon}(\tau)\in J^{\mathcal{D}_{\varepsilon}-\alpha_0\varepsilon^N}_{\varepsilon}$. Hence $\theta(\tau)=0$ and $\bar{\gamma}_{\varepsilon}(\tau)=\gamma_{\varepsilon}(\tau)$ if $\tau\in\partial[0,T]^k$. If $J_{\varepsilon}(\gamma_{\varepsilon}(\gamma_{\varepsilon}(\tau))
\le\mathcal{D}_{\varepsilon}-\alpha_0\varepsilon^N$, then $\vartheta(\tau)=0$ and so $\bar{\gamma}_{\varepsilon}(\tau) = \gamma_{\varepsilon}(\tau)\in\mathcal{X}^{\nu\varepsilon^{N/2}}_{\varepsilon}$ for large $\nu>0$. If $J_{\varepsilon}(\gamma_{\varepsilon}(\tau))
>\mathcal{D}_{\varepsilon}-\alpha_0\varepsilon^N$, then by \eqref{MMMMeq2.27}, $\gamma_{\varepsilon}(\tau)\in \mathcal{X}^{d\varepsilon^{N/2}/2}$ and by property (3) in \eqref{MMeq3.4}
$$
\mathcal{D}_{\varepsilon}-\alpha_0\varepsilon^N < J_{\varepsilon}\big(\Phi_{\varepsilon}(\gamma_{\varepsilon}(\tau),\theta)\big)\le \mathcal{D}_{\varepsilon}<\mathcal{D}_{\varepsilon} + \alpha_0\varepsilon^N,\ \ \text{for all}\ \theta\in[0,\theta(\tau)).
$$
This implies $\xi_{\varepsilon}(\varepsilon^{-N}J_{\varepsilon}(\Phi_{\varepsilon}(\gamma_{\varepsilon}(\tau_0),\theta)))\equiv 1$ for all $\theta\in[0,\theta(\tau))$. Consequently, if $\bar{\gamma}_{\varepsilon}(\tau)=\Phi_{\varepsilon}(\gamma_{\varepsilon}(\tau),\vartheta(\tau))
\not\in \mathcal{X}^{d\varepsilon^N}_{\varepsilon}$, then by the same argument of \eqref{Meq3.7}, there exists a $\theta\in(0,\theta(\tau))$ such that
$$
J_{\varepsilon}\big(\Phi_{\varepsilon}(\gamma_{\varepsilon}(\tau),\theta)\big)<\mathcal{D}_{\varepsilon}-\alpha_0\varepsilon^N.
$$
This contradicts the definition of $\theta(\tau)$. Hence $\bar{\gamma}_{\varepsilon}(\tau)\in \mathcal{X}^{d\varepsilon^{N/2}}_{\varepsilon}\subset\mathcal{X}^{\nu\varepsilon^{N/2}}_{\varepsilon}$.

Secondly, we prove that $\bar{\gamma}_{\varepsilon}(\tau)$ is continuous. We fix any $\bar{\tau}\in[0,1]^k$. If $J_{\varepsilon}(\gamma_{\varepsilon}(\bar{\tau}))<\mathcal{D}_{\varepsilon} - \alpha_0\varepsilon^N$, then $\theta(\bar{\tau})=0$. Then by the continuity of $\gamma_{\varepsilon}$, we conclude that $\bar{\gamma}_{\varepsilon}(\tau)$ is continuous at $\bar{\tau}$. If $J_{\varepsilon}(\gamma_{\varepsilon}(\bar{\tau}))=\mathcal{D}_{\varepsilon} - \alpha_0\varepsilon^N$, then from the proof of \eqref{Meq3.7}, we know that $\gamma_{\varepsilon}(\bar{\tau})\in \mathcal{X}^{d\varepsilon^{N/2}}_{\varepsilon}$, and so
$$
\|J'_{\varepsilon}\big(\gamma_{\varepsilon}(\bar{\tau})\big)\|\ge\beta\varepsilon^{N/2}>0.
$$
Thus, from the property (3) in \eqref{MMeq3.4}, we have $J_{\varepsilon}\big(\Phi_{\varepsilon}(\gamma_{\varepsilon}(\bar{\tau}),\theta(\bar{\tau})+\omega\big)
<\mathcal{D}_{\varepsilon}-\alpha_0\varepsilon^N$.
By the continuity of $\gamma_{\varepsilon}$, we choose $r>0$ as the constants such that $J_{\varepsilon}(\Phi_{\varepsilon}(\gamma_{\varepsilon}(\tau),\theta(\bar{\tau}))\big)
<\mathcal{D}_{\varepsilon}-\alpha_0\varepsilon^N$ for all $\tau\in B_r(\bar{\tau})$. Then by the definition of $\theta(\tau)$, we have $\theta({\tau})<\theta(\bar{\tau})$ for all $\tau\in B_r(\bar{\tau})\cap[0,T]^k$, and then
$$
0\le\limsup_{\tau\to\bar{\tau}}\theta(\tau)\le\theta(\bar{\tau}).
$$
If $\theta({\bar{\tau}})=0$, we immediately have
$$
\lim_{\tau\to\bar{\tau}}\theta(\tau)=\theta(\bar{\tau}).
$$
If $\theta({\tau})>0$, then for any $0<\omega<\theta(\bar{\tau})$, similarly we have $J_{\varepsilon}(\Phi_{\varepsilon}(\gamma_{\varepsilon}(\tau),\theta(\bar{\tau})-\omega)\big)
>\mathcal{D}_{\varepsilon}-\alpha_0\varepsilon^N$. By the continuity of $\gamma_{\varepsilon}$ again, we see that
$$
\liminf_{\tau\to\bar{\tau}}\theta(\tau)\ge\theta(\bar{\tau}).
$$
So $\theta({\cdot})$ is continuous at $\bar{\tau}$. This completes the proof of Claim 2.

Now we have proved that $\bar{\gamma}_{\varepsilon}(\tau)\in \Psi_{\varepsilon}$ and $\max_{\tau\in[0,T]^k}\le \mathcal{D}_{\varepsilon}-\alpha_0\varepsilon^N$, which contradicts the definition of $\mathcal{C}_{\varepsilon}$. This completes the proof.

\end{proof}

\begin{lemma}\label{le2.11}
Let $\{u_n\}^{\infty}_{n=1}$ be the sequence given by Proposition \ref{Mpr3.3}. Then $\{u_n\}$ has a subsequence which converges to $u_{\varepsilon}$ in $\mathcal{D}^s_{V,\varepsilon}(\R^N)$. Moreover, there hold $u_{\varepsilon}>0$, $u_{\varepsilon} \in \mathcal{D}^s_{V,\varepsilon}(\mathbb{R}^N)\cap C^{1,\beta}(\mathbb{R}^N)$ for some $\beta\in (0,1)$ and $u_{\varepsilon}$ is a solution to the penalized problem \eqref{eq2.2}(or \eqref{BBBeq2.2}).
\end{lemma}

\begin{proof}
The convergence is from Lemma \ref{le2.4}. The regularity result follows from Appendix D in \cite{20}. Testing the penalized equation \eqref{BBBeq2.2} with $(u_{\varepsilon})_{-}$ and integrating, we can see that $u_{\varepsilon}\geq 0$. Suppose to the contrary that there exists  $x_0\in \mathbb{R}^N$ such that $u_{\varepsilon}(x_0) = 0$, then we have
$$
0 = \varepsilon^{2s}(-\Delta)^s{u_{\varepsilon}}(x_0) + V(x_0)u_{\varepsilon}(x_0) < 0,
$$
which is a contradiction. Therefore, $u_{\varepsilon}>0$.

\end{proof}

To end this section,  we prove that $u_{\varepsilon}$ owns $k$-peaks.
\begin{lemma}\label{Mle3.5}
 Let $\rho > 0$ and $u_{\varepsilon}$ be the solution of \eqref{eq2.2} given by Lemma \ref{le2.11}. Then there exists $k$ families of points $\{x^i_{\varepsilon}\}$, $i=1,\ldots,k$, such that
\begin{eqnarray*}
%  to remove numbering (before each equation)
  &&(1)\ \liminf_{\varepsilon\to 0}\|u_{\varepsilon}\|_{L^{\infty}(B_{\varepsilon\rho}(x^i_{\varepsilon}))}> 0,\\
  &&(2)\ \lim_{\varepsilon\to 0}dist(x^i_{\varepsilon},\mathcal{M}_i)=0,\\
  &&(3)\ \lim\limits_{{R\to \infty}\atop{{\varepsilon\to 0}}}\|u_{\varepsilon}\|_{L^{\infty}(U\backslash \cup_{1\leq i\leq k} B_{\varepsilon R}(x^i_{\varepsilon}))} =  0.
\end{eqnarray*}

\end{lemma}

\begin{proof}
The proof  is trivial by the fact that the $(PS)$ sequence given by Proposition \ref{Mpr3.3} satisfies the assumptions of Proposition \ref{Mpr2.10}.
\end{proof}

\section{Back to the original problem}\label{s4}

\noindent In this section we show that $u_{\varepsilon}$ solves the original problem \eqref{eq1.1}. For this purpose, basing on the penalized equation \eqref{BBBeq2.2}, all we need to do is to prove that
\begin{equation}\label{Meq3.1}
f(u_{\varepsilon})\leq \mathcal{P}_{\varepsilon}(x)u_{\varepsilon},\ \ x\in\R^N\backslash\Lambda.
\end{equation}
We use comparison principle to prove \eqref{Meq3.1}, for which we should first linearize the penalized equation \eqref{BBBeq2.2} outside small balls.
\begin{proposition}\label{pr4.1} Let $\{x^i_{\varepsilon}\}, i=1,\ldots,k$ be the $k$ families of points given by Lemma \ref{Mle3.5}. Then for $\varepsilon > 0$ small enough and $\delta\in (0,1)$, there exist $C_{\infty}>0$ and $R > 0$ such that
\begin{equation}\label{Meq3.2}
  \left\{
    \begin{array}{ll}
      \varepsilon^{2s}(-\Delta)^{s} u_{\varepsilon} + (1 - \delta)Vu_{\varepsilon}\leq P_{\varepsilon}u_{\varepsilon}, & \text{in}\ \mathbb{R}^N\backslash  \bigcup_{i = 1}^kB_{R\varepsilon}(x^i_{\varepsilon}), \vspace{2mm}\\
      u_{\varepsilon}\leq C_{\infty} & \text{in}\ \Lambda.
    \end{array}
  \right.
\end{equation}

\end{proposition}

\begin{proof}
 That $u_{\varepsilon}\leq C_{\infty}$ in $\Lambda$ is from Lemma \ref{Mle3.5} and the $L^{\infty}$ estimate in \cite[Appendix D]{20}. By the assumption on $f$, $\inf_{U} V(x) > 0$  and Lemma \ref{Mle3.5}, there exists $R > 0$ such that
$$
f(u_{\varepsilon})\leq \delta Vu_{\varepsilon}\ \text{in}\ U\backslash \bigcup_{i = 1}^kB_{R\varepsilon }(x^i_{\varepsilon}).
$$
Obviously
$$
\mathfrak{g}_{\varepsilon}(u_{\varepsilon})
\leq
\mathcal{P}_{\varepsilon}u_{\varepsilon}\ \text{in}\ \mathbb{R}^N\backslash U.
$$
Hence we conclude our result by inserting the previous pointwise bounds into the penalized equation \eqref{BBBeq2.2}.

\end{proof}

Next, we construct a suitable sup-solution to Eq \eqref{Meq3.2}. Some of the the details are similar to that in Proposition 4.2 of \cite{ADP-AA-2021}. Let $\tilde\eta_{\beta}(s), s\ge 0$ be a smooth non-increasing function with $\tilde{\eta}_{\beta}\equiv 1$ on $[0,1]$ and $\tilde{\eta}_{\beta}\equiv 0$ on $(1+\beta,+\infty)$, where $\beta$ is a small parameter. Define $\eta_{\beta,R}(|x|)= \tilde{\eta}_{\beta}(|x|/R)$.
Setting $0<\alpha <N-2s$ and denoting
\begin{align*}
&\quad f^{\alpha}_{\beta,R}(x)
= \eta_{\beta,R}(x)\frac{1}{R^{\alpha}} +\big(1 -
\eta_{\beta,R}(x)\big)\frac{1}{|x|^{\alpha}},\\
&\quad f^{\alpha,i}_{\beta,R,\varepsilon}(x) = f^{\alpha}_{\beta,R}\Big(\frac{x-x^i_{\varepsilon}}{\varepsilon}\Big), \\
&\quad  f^{\alpha}_{\beta,R,\varepsilon}(x)=\sum_{i=1}^kf^{\alpha,i}_{\beta,R,\varepsilon}(x).
\end{align*}
We have
\begin{proposition}\label{Mpr3.2}
Let $\varepsilon>0$ be small enough. Then for every $x\in \R^N\backslash \bigcup_{i=1}^kB_{R\varepsilon}(x^i_{\varepsilon})$, it holds
\begin{equation}\label{Meq3.3}
\varepsilon^{2s}(-\Delta)^s f^{\alpha}_{\beta,R,\varepsilon} + (1-\delta)V(x) f^{\alpha}_{\beta,R,\varepsilon} - \mathcal{P}_{\varepsilon}(x) f^{\alpha}_{\beta,R,\varepsilon}\ge 0.
\end{equation}
\end{proposition}
\begin{proof}
Fixing any $i\in\{1,\ldots,k\}$, a computation shows that
\begin{align}
\nonumber&\quad \varepsilon^{2s}(-\Delta)^s f^{\alpha,i}_{\beta,R,\varepsilon} + V(x) f^{\alpha,i}_{\beta,R,\varepsilon} - \mathcal{P}_{\varepsilon}(x) f^{\alpha,i}_{\beta,R,\varepsilon}\\
&= (-\Delta)^s f^{\alpha}_{\beta,R,\varepsilon}\Big(\frac{x-x^i_{\varepsilon}}{\varepsilon}\Big) + V(x) f^{\alpha}_{\beta,R,\varepsilon}\Big(\frac{x-x^i_{\varepsilon}}{\varepsilon}\Big) - \mathcal{P}_{\varepsilon}(x) f^{\alpha}_{\beta,R,\varepsilon}\Big(\frac{x-x^i_{\varepsilon}}{\varepsilon}\Big)\\
\nonumber &= \Big((-\Delta)^s f^{\alpha}_{\beta,R,\varepsilon}(y) + V^i_{\varepsilon}(y) f^{\alpha}_{\beta,R,\varepsilon}(y) - \widehat{\mathcal{P}}^i_{\varepsilon}(y) f^{\alpha}_{\beta,R,\varepsilon}(y)\Big)\Big|_{y=\frac{x-x^i_{\varepsilon}}{\varepsilon}},
\end{align}
where $V^i_{\varepsilon}(\cdot) = V(\varepsilon x\cdot + x^i_{\varepsilon})$ and $\widehat{\mathcal{P}}^i_{\varepsilon}(\cdot) = \mathcal{P}_{\varepsilon}(\varepsilon \cdot + x^i_{\varepsilon})$. But,  using the non-increasing property of $\eta_{\beta}$ and the computation of Proposition 4.2 of \cite{ADP-AA-2021},  for any $y\in\R^N\backslash B_R(0)$,  when $\varepsilon>0$ is small enough, we can conclude that
\begin{equation}\label{Meq3.5}
(-\Delta)^s f^{\alpha}_{\beta,R,\varepsilon}(y) + V^i_{\varepsilon}(y) f^{\alpha}_{\beta,R,\varepsilon}(y) - \widehat{\mathcal{P}}^i_{\varepsilon}(y) f^{\alpha}_{\beta,R,\varepsilon}(y)\ge 0.
\end{equation}
Then for all $x\in \R^N\backslash B_{R\varepsilon}(x^i_{\varepsilon})$, it holds
\begin{align}
\nonumber&\quad \varepsilon^{2s}(-\Delta)^s f^{\alpha,i}_{\beta,R,\varepsilon} + V(x) f^{\alpha,i}_{\beta,R,\varepsilon} - \mathcal{P}_{\varepsilon}(x) f^{\alpha,i}_{\beta,R,\varepsilon}\ge 0.
\end{align}
As a result, we have
\begin{align*}\label{Meq3.3}
&\quad \varepsilon^{2s}(-\Delta)^s f^{\alpha}_{\beta,R,\varepsilon} + V(x) f^{\alpha}_{\beta,R,\varepsilon} - \mathcal{P}_{\varepsilon}(x) f^{\alpha}_{\beta,R,\varepsilon}\\
&=\sum_{i=1}^k\Big(\varepsilon^{2s}(-\Delta)^s f^{\alpha,i}_{\beta,R,\varepsilon} + V(x) f^{\alpha,i}_{\beta,R,\varepsilon} - \mathcal{P}_{\varepsilon}(x) f^{\alpha,i}_{\beta,R,\varepsilon}\Big)\ge 0
\end{align*}
for all $x\in \R^N\backslash \bigcup_{i=1}^kB_{R\varepsilon}(x^i_{\varepsilon})$. This completes the proof.
\end{proof}

At last, we give the proof of Theorem \ref{th1.1}.

\textbf{Proof of Theorem \ref{th1.1}}.
Let
\begin{equation}\label{eq4.5}
\left\{
  \begin{array}{ll}
    %\alpha \in\Big(\frac{2s}{p - 2},N - 2s\Big),\ \kappa = \frac{\alpha(p - 2) - 2s}{4},
\mathcal{P}_{\varepsilon}(x) = \frac{\varepsilon^{2s + 2\kappa}}{|x|^{2s + \kappa}}\chi_{\R^N\backslash\Lambda}(x), & \\[5mm]
    \overline{U}_{\varepsilon}(x) = CR^{\alpha}f^{\alpha}_{\beta,R,\varepsilon}(x).
  \end{array}
\right.
\end{equation}

 It is easy to check that $\mathcal{P}_{\varepsilon}$ satisfies the assumption \eqref{eq2.1}.

 By Proposition \ref{Mpr3.2}, choosing the constant  $C>0$ large enough and letting  $v_{\varepsilon}(x) =  u_{\varepsilon}(x)-\overline{U}_{\varepsilon}(x)$, we have
\begin{equation*}
  \left\{
    \begin{array}{ll}
      \varepsilon^{2s} (-\Delta)^{s} {v}_{\varepsilon}(x) + (1-\delta)V(x){v}_{\varepsilon}(x) - \mathcal{P}_{\varepsilon}(x){v}_{\varepsilon}(x)\le 0, & \text{in}\  \R^N\backslash \bigcup_{i=1}^kB_{R\varepsilon}(x^i_{\varepsilon}), \vspace{2mm}\\[5mm]
      {v}_{\varepsilon}(x)\le 0 & \text{in}\ \bigcup_{i=1}^kB_{R\varepsilon}(x^i_{\varepsilon}).
    \end{array}
  \right.
\end{equation*}
Since ${v}^+_{\varepsilon}\in \mathcal{D}^s_{V,\varepsilon}$(when $\alpha$ is closed to $N - 2s$), testing the equation above against with ${v}^+_{\varepsilon}(x)$, by the fractional Hardy inequality in \eqref{eq1.4}, we find ${v}^+_{\varepsilon}(x) = 0,\ x\in\R^N$. Hence ${v}_{\varepsilon}(x)\le 0,\ x\in\R^N$. Especially, we have
$$
u_{\varepsilon}(x)\le \overline{U}_{\varepsilon}(x) = \sum_{i=1}^kf^{\alpha,i}_{\beta,R,\varepsilon}(x)\le \sum_{i=1}^k\frac{C\varepsilon^{\alpha}}{\varepsilon^{\alpha} + |x-x^i_{\varepsilon}|^{\alpha}},\ \ x\in\R^N.
$$
Moreover,  letting $\alpha$ be closed to $N-2s$, for all $x\in\R^N\backslash\Lambda$, it holds
%$$
%\big(u_{\varepsilon}(x)\big)^{p - 2}\le \mathcal{P}_{\varepsilon}(x),\ \forall  \, x\in\R^N\backslash\Lambda.
%$$
\begin{align*}
 \frac{f({u}_{\epsilon})}{{u}_{\epsilon}}&\le ({u}_{\epsilon})^{\tilde{\kappa}}\le \frac{C\epsilon^{\alpha\tilde{\kappa}}}{|x|^{\alpha\tilde{\kappa}}}\le \frac{\epsilon^{2s + 2\kappa}}{|x|^{2s + \kappa}} = \mathcal{P}_{\epsilon}(x).
\end{align*}
This gives \eqref{Meq3.1}. As a result, $u_{\varepsilon}$ solves the original problem.
\begin{remark}\label{re4.8}
In the local case $s = 1$, we can prove the same result more easily by introducing the same penalized function $\mathcal{P}_{\varepsilon}$ in this paper. We point out here that we also answer positively to the conjecture proposed by Ambrosetti and Malchiodi in \cite{AA} in the nonlocal case.
\end{remark}

\appendix
  \renewcommand{\appendixname}{Appendix~\Alph{section}}
  \section{Appendix}
\noindent In this section we are going to verify Lemma \ref{le2.8}, %\eqref{addeq2.22},
\eqref{MMMeq2.26}, \eqref{MMMeq2.28}, \eqref{addeq2.27} and \eqref{MMMeq2.29}.
\begin{proposition}\label{prA.1} For every $i = 1,\ldots,k$, it holds
$$
\lim_{\varepsilon\to 0}\frac{c^i_{\varepsilon}}{\varepsilon^N} = c_{\lambda_i}.
$$
\end{proposition}
\begin{proof}
The achievement of $c^i_{\varepsilon}$ is easily from the fact that the embedding
$$
W^{s,2}(\Omega)\hookrightarrow L^p
$$
is compact for  $1 \leq p < 2^*_s$(see \cite{4} for more details). Thus we only need to prove \eqref{eq2.3}.

For every nonnegative $v\in C^{\infty}_c(\mathbb{R}^N)\backslash\{0\}$ and $x_0\in\Lambda_i$, let $v_{\varepsilon}(x) = v\big(\frac{x - x_0}{\varepsilon}\big)$. Obviously, $supp\ v_{\varepsilon}\subset\Lambda_i$ and $\gamma(t) = tTv_{\varepsilon}\in \Gamma^i_{\varepsilon}$ for $\varepsilon$ small enough and $T$ large enough. Therefore,
\begin{align*}
  c^i_{\varepsilon}&\leq  \max_{t\in[0,1]}J^i_{\varepsilon}(\gamma(t))\\
                   &\leq \varepsilon^N\max_{t > 0}\Big(\frac{t^2}{2}\int_{\mathbb{R}^N}\int_{\mathbb{R}^N}\frac{|v(x) - v(y)|^2}{|x - y|^{N + 2s}}dxdy\\
                   &\quad+ \frac{t^2}{2}\int_{\mathbb{R}^N}V(\varepsilon x + x_0)|v|^2 dx - \int_{\mathbb{R}^N}F(tv)dx\Big)
\end{align*}
and then
\begin{align*}
 \limsup_{\varepsilon\to 0} \frac{c^i_{\varepsilon}}{\varepsilon^N}
                      &\leq \limsup_{\varepsilon\to 0}\max_{t > 0}\Big(\frac{t^2}{2}\int_{\mathbb{R}^N}\int_{\mathbb{R}^N}\frac{|v(x) - v(y)|^2}{|x - y|^{N + 2s}}dxdy\\
                      &\qquad+ \frac{t^2}{2}\int_{\mathbb{R}^N}V(\varepsilon x + x_0)|v|^2 dx - \int_{\mathbb{R}^N}F(tv)dx\Big)\\
                      & = \max_{t > 0}L_{v(x_0)}(tv).
\end{align*}
Hence, by the arbitrariness  of $v$ and $x_0$, we have
\begin{eqnarray}\begin{split}\label{eqA.1}
 \limsup_{\varepsilon\to 0} \frac{c^i_{\varepsilon}}{\varepsilon^N}&\leq c_{\lambda_i}.
\end{split}\end{eqnarray}

On the other hand, let $w_{\varepsilon}$ be a critical point corresponding to $c^i_{\varepsilon}$, i.e., $J^i_{\varepsilon}(w_{\varepsilon}) = c^i_{\varepsilon}$ and
\begin{equation}\label{eqA.2}
\varepsilon^{2s}\int_{ S_i}\frac{w_{\varepsilon}(x) - w_{\varepsilon}(y)}{|x - y|^{N + 2s}}dy + V(x)w_{\varepsilon}(x) = \mathfrak{g}_{\varepsilon}(w_{\varepsilon}),\ \  x\in S_i.
\end{equation}
It follows that
\begin{equation*}
\varepsilon^{2s}\int_{ S_i}\int_{ S_i}\frac{w_{\varepsilon}(x) - w_{\varepsilon}(y)}{|x - y|^{N + 2s}}w_{\varepsilon}(x)dydx + \int_{S_i}V(x)|w_{\varepsilon}(x)|^2 = \int_{S_i}\mathfrak{g}_{\varepsilon}(w_{\varepsilon})w_{\varepsilon}.
\end{equation*}
Then by \eqref{MMMeq2.2}, it holds
\begin{align*}
&\quad\quad\varepsilon^{2s}\int_{ S_i}\int_{ S_i}\frac{w_{\varepsilon}(x) - w_{\varepsilon}(y)}{|x - y|^{N + 2s}}w_{\varepsilon}(x)dydx + \int_{S_i}V(x)|w_{\varepsilon}(x)|^2\\
 &\leq C\big(\|w_{\varepsilon}\|^{p - 1}_{L^{\infty}(\Lambda_i)} + \|w_{\varepsilon}\|^{\tilde{\kappa}}_{L^{\infty}(\Lambda_i)}\big)\Big(\varepsilon^{2s}\int_{ S_i}\int_{ S_i}\frac{w_{\varepsilon}(x) - w_{\varepsilon}(y)}{|x - y|^{N + 2s}}w_{\varepsilon}(x)dydx + \int_{S_i}V(x)|w_{\varepsilon}(x)|^2\Big),
\end{align*}
from which %and using the same proof of $(1)$ in Lemma \ref{le3.3},
we conclude that there exists $x^i_{\varepsilon}\in\overline{\Lambda_i}$ such that for $\rho > 0$,
\begin{equation}\label{eqA.3}
  \liminf_{\varepsilon\to 0}\|w_{\varepsilon}\|_{L^{\infty}(B_{\varepsilon\rho}(x^i_{\varepsilon}))} > 0.
\end{equation}
Going if necessary to a subsequence, we assume that
\begin{equation}\label{eqA.4}
  \lim_{\varepsilon\to 0}x^i_{\varepsilon}\to x^i\in\overline{\Lambda_i}.
\end{equation}

Now, let $\tilde{w}_{\varepsilon}(x) = w_{\varepsilon}(x^i_{\varepsilon} + \varepsilon x)$, then $\tilde{w}_{\varepsilon}$ satisfies
\begin{equation}\label{eqA.5}
\int_{ S^i_{\varepsilon}}\frac{\tilde{w}_{\varepsilon}(x) - \tilde{w}_{\varepsilon}(y)}{|x - y|^{N + 2s}}dy + V_{\varepsilon}(x)\tilde{w}_{\varepsilon}(x) = \tilde{\mathfrak{g}}_{\varepsilon}(\tilde{w}_{\varepsilon})
\ \ x\in S^i_{\varepsilon},
\end{equation}
where $V_{\varepsilon}(x) = V(x^i_{\varepsilon} + \varepsilon x)$, $S^i_{\varepsilon} = \{x\in\mathbb{R}^N:\varepsilon x + x^i_{\varepsilon}\in S\}$ and $\tilde{\mathfrak{g}}_{\varepsilon}(\tilde{w}_{\varepsilon}) = g(\varepsilon x + x^i_{\varepsilon},\tilde{w}_{\varepsilon})$. Moreover, by \eqref{eqA.1}, we have
$$
\sup_{\varepsilon > 0}\|\tilde{w}_{\varepsilon}\|_{W^{s,2}(B_R)} < \infty
$$
for every $R\in(0,+\infty)$. Thus, by diagonal argument, we conclude that $\tilde{w}_{\varepsilon}\rightharpoonup \tilde{w}$ weakly in $W^{s,2}(B_R)$ for every $R>0$. Moreover, it is easy to check by Fatou's Lemma that $\tilde{w}\in H^s(\mathbb{R}^N)$.
Then, by \eqref{eqA.4}, using Corollary 7.2 in \cite{4} and  taking  limit in \eqref{eqA.5}, we conclude that
\begin{equation*}
 \int_{\mathbb{R}^N}\frac{\tilde{w}(x) - \tilde{w}(y)}{|x - y|^{N + 2s}}dy + V(x^i)\tilde{w} = \chi_{\Lambda^i_*}f(\tilde{w})\ \ x\in \mathbb{R}^N,
\end{equation*}
where $\Lambda^*_i$ is the limit of the set $\Lambda^i_{\varepsilon} = \{x\in\mathbb{R}^N:\varepsilon x + x^i_{\varepsilon}\in\Lambda_i\}$. But by \eqref{eqA.3} and using the standard bootstrap argument in Appendix D in \cite{20}, we have
$$
\|\tilde{w}\|_{L^{\infty}(B_\rho(0))} = \lim_{\varepsilon\to 0}\|\tilde{w}_{\varepsilon}\|_{L^{\infty}(B_\rho(0))}\geq\liminf_{\varepsilon\to 0}\|{w}_{\varepsilon}\|_{L^{\infty}(B_\rho(0))} > 0,
$$
which combined  with  the Liouville-type results (see Lemma 3.3 in \cite{APX}) implies that $\Lambda^i_* = \mathbb{R}^N$. Hence we have
$$
(-\Delta)^s\tilde{w} + V(x^i)\tilde{w} = f(\tilde{w})\ \ \text{in}\ \R^N.
$$
Proceeding as one proves  Lemma 3.3 of \cite{ADP-AA-2021}, we have
\begin{align*}
 \liminf_{\varepsilon\to 0} \frac{c^i_{\varepsilon}}{\varepsilon^N}
& \geq L_{V(x^i)}(\tilde{w}) + o_R(1)\\
&\quad + \liminf_{\varepsilon\to 0}\frac{1}{\varepsilon^N}\Big(\frac{1}{2}\int_{S^i_{\varepsilon}\backslash B_R}dx
\int_{S^i_{\varepsilon}}
 \frac{|\tilde{w}_{\varepsilon}(x) - \tilde{w}_{\varepsilon}(y)|^2}{|x - y|^{N + 2s}}dy\\
&\quad + \frac{1}{2}\int_{S^i_{\varepsilon}\backslash B_R}V_{\varepsilon}(x)\tilde{w}^2_{\varepsilon}(x) dx - \int_{S^i_{\varepsilon}\backslash B_R}\widetilde{\mathfrak{G}}_{\varepsilon}(\tilde{w}_{\varepsilon}(x))dx\Big)\\
&\geq  c_{V(x^i)} + o_R(1)
\end{align*}
% Moreover, using equations (3.8)-(3.11) in \cite{APX}, we have
%\begin{equation}\label{eqA.7}
%|\liminf_{\varepsilon\to 0}T_{\varepsilon,R}| = o_R(1).
%\end{equation}
%Letting $R\to\infty$, we see
%\begin{align*}
% \liminf_{\varepsilon\to 0} \frac{c^i_{\varepsilon}}{\varepsilon^N}\geq c_{V(x^i)}.
%\end{align*}
 Therefore,
\begin{equation*}\label{eq2.10}
\liminf_{\varepsilon\to 0}\frac{c^i_{\varepsilon}}{\varepsilon^N}\geq c_{\lambda_i},
\end{equation*}
which and  \eqref{eqA.1} complete the proof.
\end{proof}

\begin{lemma}\label{leA.2} The estimates %\eqref{addeq2.22},
\eqref{MMMeq2.26}, \eqref{MMMeq2.28}, \eqref{addeq2.27} and \eqref{MMMeq2.29} hold.

\end{lemma}

\begin{proof}
Hereafter, we define $\hat{\eta}_{\varepsilon}(x)=\eta(2\varepsilon x)=\eta_{\varepsilon}(2x)$ for all $x\in\R^N$.
We first give the proof of \eqref{MMMeq2.28}. By the definition of $\bar{\eta}_{\varepsilon}$, we have
\begin{align*}
2T^2_{\varepsilon}(\tilde{\eta}_{\varepsilon})/\varepsilon^N &=\sum_{i=1}^k\varepsilon^{2s-N}\int_{\R^N}dx\int_{\R^N}(u_{\varepsilon}(x) -u_{\varepsilon}(y))\Big(\eta\big(2(x-p^i_{\varepsilon}-\varepsilon z_i)\big)\\
&\qquad-\eta\big(2(y-p^i_{\varepsilon}-\varepsilon z_i)\big)\Big)u_{\varepsilon}(y){|x - y|^{-N - 2s}}dy\\
&=\sum_{i=1}^k\int_{\R^N}dx\int_{\R^N}\frac{(v^i_{\varepsilon}(x) - v^i_{\varepsilon}(y))(\hat{\eta}_{\varepsilon}(x)-\hat{\eta}_{\varepsilon}(y))v^i_{\varepsilon}(y)}{|x - y|^{N + 2s}}dy\\
&:=\sum_{i=1}^kT^{2,i}_{\varepsilon}(\eta).
\end{align*}
For each $i=1,\ldots,k$, dividing $\mathbb{R}^N$ into several regions, we have
\begin{align*}
T^{2,i}_{\varepsilon}(\eta)
&= \int_{B_{\frac{\beta}{\varepsilon}}}dx\int_{B_{\frac{\beta}{\varepsilon}}}\frac{(v^i_{\varepsilon}(x) - v^i_{\varepsilon}(y))(\hat{\eta}_{\varepsilon}(x)-\hat{\eta}_{\varepsilon}(y))v^i_{\varepsilon}(y)}{|x - y|^{N + 2s}}dy\\
&\quad + \int_{B_{\frac{\beta}{\varepsilon}}}dx\int_{B^c_{\frac{\beta}{\varepsilon}}}\frac{(v^i_{\varepsilon}(x) - v^i_{\varepsilon}(y))(\hat{\eta}_{\varepsilon}(x)-\hat{\eta}_{\varepsilon}(y))v^i_{\varepsilon}(y)}{|x - y|^{N + 2s}}dy\\
& \quad + \int_{B^c_{\frac{\beta}{\varepsilon}}}dx\int_{B_{\frac{\beta}{\varepsilon}}}\frac{(v^i_{\varepsilon}(x) - v^i_{\varepsilon}(y))(\hat{\eta}_{\varepsilon}(x)-\hat{\eta}_{\varepsilon}(y))v^i_{\varepsilon}(y)}{|x - y|^{N + 2s}}dy\\
&:=\sum_{j=1}^3 T^{2,i,j}_{\varepsilon}(\eta).
\end{align*}
For $T^{2,i,1}_{\varepsilon}(\eta)$, by Cauchy inequality, we  have
\begin{align*}
|T^{2,i,1}_{\varepsilon}(\eta)|^2
&\le  C\int_{B_{\frac{\beta}{\varepsilon}}}|v^i_{\varepsilon}(y)|^2dy\int_{B_{\frac{\beta}{\varepsilon}}}\frac{|\hat{\eta}_{\varepsilon}(x)-\hat{\eta}_{\varepsilon}(y)|^2}{|x - y|^{N + 2s}}dx\\
&\le C\varepsilon^2\int_{B_{\frac{\beta}{\varepsilon}}}|v^i_{\varepsilon}(z)|^2dy\int_{B_{\frac{2\beta}{\varepsilon}}}\frac{1}{|z|^{N + 2s-2}}dx\\
&=C\varepsilon^{2s}.
\end{align*}
For $T^{2,i,2}_{\varepsilon}(\eta)$, by the definition of $\eta$, we  have
\begin{align*}
T^{2,i,2}_{\varepsilon}(\eta)
&\le \int_{B_{\frac{\beta}{\varepsilon}}}dx\int_{B^c_{\frac{\beta}{\varepsilon}}}\frac{ v^i_{\varepsilon}(x)\hat{\eta}_{\varepsilon}(x)v^i_{\varepsilon}(y)}{|x - y|^{N + 2s}}dy.
%&\quad  + \int_{B_{\frac{\beta}{\varepsilon}}}dx\int_{B_{\frac{3\beta}{\varepsilon}}\backslash B_{\frac{\beta}{\varepsilon}}}\frac{(v^i_{\varepsilon}(y) - v^i_{\varepsilon}(x))(\hat{\eta}_{\varepsilon}(x)-\hat{\eta}_{\varepsilon}(y))v^i_{\varepsilon}(y)}{|x - y|^{N + 2s}}dy\Big|\\
%&\le \int_{B_{\frac{\beta}{\varepsilon}}}dx\int_{B^c_{\frac{3\beta}{\varepsilon}}}\frac{|v^i_{\varepsilon}(y) - v^i_{\varepsilon}(x)|\hat{\eta}_{\varepsilon}(x)|v^i_{\varepsilon}(y)|}{|x - y|^{N + 2s}}dy+C\varepsilon^s
\end{align*}
But, using the similar estimate of $T^{2,i,1}_{\varepsilon}(\eta)$ and fractional Hardy inequality \eqref{eq1.4}, we have
\begin{align*} &\quad\Big|\int_{B_{\frac{\beta}{\varepsilon}}}dx\int_{B^c_{\frac{\beta}{\varepsilon}}}\frac{ v^i_{\varepsilon}(x)\hat{\eta}_{\varepsilon}(x)v^i_{\varepsilon}(y)}{|x - y|^{N + 2s}}dy\Big|\\
&\le \int_{B_{\frac{\beta}{\varepsilon}}}dx\int_{B^c_{\frac{3\beta}{\varepsilon}}}\frac{ |v^i_{\varepsilon}(x)|\hat{\eta}_{\varepsilon}(x)|v^i_{\varepsilon}(y)|}{|x - y|^{N + 2s}}dy\\
&\quad + \int_{B_{\frac{\beta}{\varepsilon}}}dx\int_{B_{\frac{3\beta}{\varepsilon}}\backslash B_{\frac{\beta}{\varepsilon}}}\frac{ |v^i_{\varepsilon}(x)||\hat{\eta}_{\varepsilon}(x)-\hat{\eta}_{\varepsilon}(y)||v^i_{\varepsilon}(y)|}{|x - y|^{N + 2s}}dy\Big|\\
&\le\int_{B_{\frac{\beta}{\varepsilon}}}dx\int_{B^c_{\frac{3\beta}{\varepsilon}}}\frac{ |v^i_{\varepsilon}(x)|\hat{\eta}_{\varepsilon}(x)|v^i_{\varepsilon}(y)|}{|x - y|^{N + 2s}}dy + C\varepsilon^s\\
&\le \Big(\int_{B_{\frac{\beta}{\varepsilon}}}(\hat{\eta}_{\varepsilon}(x)v^i_{\varepsilon}(x))^2dx
\int_{B^c_{\frac{3\beta}{\varepsilon}}}\frac{1}{|x - y|^{N + 2s}}dy\Big)^{\frac{1}{2}}\\
&\quad\cdot\Big(\int_{B^c_{\frac{3\beta}{\varepsilon}}}\frac{\big(v^i_{\varepsilon}(y)\big)^2}{|y|^{2s}}dy
\int_{B_{\frac{\beta}{\varepsilon}}}\frac{|y|^{2s}}{|x - y|^{N + 2s}}dx\Big)^{\frac{1}{2}} + C\varepsilon^s\\
&\le C\varepsilon^s.
\end{align*}
Hence, it holds
\begin{equation*}
\limsup_{\varepsilon\to 0}T^{2,i,2}_{\varepsilon}(\eta)\le0.
\end{equation*}
Similarly, one has
\begin{equation*}
\limsup_{\varepsilon\to 0}T^{2,i,3}_{\varepsilon}(\eta)\le0.
\end{equation*}
So \begin{equation*}
\limsup_{\varepsilon\to 0}T^{2,i}_{\varepsilon}(\eta)\le0
\end{equation*}
and \begin{equation*}
\limsup_{\varepsilon\to 0}\frac{T^{2}_{\varepsilon}(\eta)}{\varepsilon^N}\le0.
\end{equation*}

Secondly, we prove \eqref{MMMeq2.26}. By the definition of $\eta$, we have
\begin{align*}
|T^{1}_{\varepsilon}(\eta)/2|^2
&\le\varepsilon^{4s}\Big(\int_{B_{\beta}(p^i_{\varepsilon}+\varepsilon z_i)}(u_{\varepsilon}(x))^2dx\int_{B_{\beta}(p^j_{\varepsilon}+\varepsilon z_j)}\frac{1}{|x - y|^{N+2s}}dy\Big)\\
&\cdot \Big(\int_{B_{\beta}(p^j_{\varepsilon}+\varepsilon z_j)}(u_{\varepsilon}(y))^2dy\int_{B_{\beta}(p^j_{\varepsilon}+\varepsilon z_j)}\frac{1}{|x - y|^{N+2s}}dx\Big)\\
&= \varepsilon^{4N+4s} \Big(\int_{B_{\frac{\beta}{\varepsilon}}}(v^i_{\varepsilon}(x))^2dx\int_{B_{\frac{\beta}{\varepsilon}}}\frac{1}{|\varepsilon x + p^i_{\varepsilon}+\varepsilon z_i - \varepsilon y - p^j_{\varepsilon}-\varepsilon z_j|^{N+2s}}dy\Big)\\
&\cdot \Big(\int_{B_{\frac{\beta}{\varepsilon}}}(v^j_{\varepsilon}(y))^2dy\int_{B_{\frac{\beta}{\varepsilon}}}\frac{1}{|\varepsilon x + p^i_{\varepsilon}+\varepsilon z_i- \varepsilon y-p^j_{\varepsilon}-\varepsilon z_j|^{N+2s}}dx\Big)\\
&= \varepsilon^{2N} \Big(\int_{B_{\frac{\beta}{\varepsilon}}}(v^i_{\varepsilon}(x))^2dx\int_{B_{\frac{\beta}{\varepsilon}}}\frac{1}{|x - y + \frac{p^i_{\varepsilon}+\varepsilon z_i -  p^j_{\varepsilon}+\varepsilon z_j}{\varepsilon}|^{N+2s}}dy\Big)\\
&\cdot \Big(\int_{B_{\frac{\beta}{\varepsilon}}}(v^j_{\varepsilon}(y))^2dy\int_{B_{\frac{\beta}{\varepsilon}}}\frac{1}{|x - y + \frac{p^i_{\varepsilon}+\varepsilon z_i -  p^j_{\varepsilon}-\varepsilon z_j}{\varepsilon}|^{N+2s}}dx\Big)\\
&\le C\varepsilon^{2N+4s}.
\end{align*}
Then we have
$$
\frac{T^1_{\varepsilon}(\eta)}{\varepsilon^N}\le C\varepsilon^s,
$$
which gives \eqref{MMMeq2.26}.

Thirdly, we give the proof of \eqref{MMMeq2.29}. Denoting $A_{\varepsilon}=\R^N\backslash\bigcup_{i=1}^kB_{2\beta}(p^i_{\varepsilon}+\varepsilon z_i)$,  one can check  that
\begin{align*}
\varepsilon^{-2s}T^{3}_{\varepsilon}(\breve{\eta})&=\int_{\R^N}\int_{\R^N}\frac{|\breve{\eta}_{\varepsilon}(x)u_{\varepsilon}(x) - \breve{\eta}_{\varepsilon}(y)u_{\varepsilon}(y)|^2}{|x-y|^{N+2s}}dxdy\\
&= \int_{A_{\varepsilon}}dx\int_{A_{\varepsilon}}\frac{|u_{\varepsilon}(x) - u_{\varepsilon}(y)|^2}{|x-y|^{N+2s}}dxdy + \int_{A_{\varepsilon}}dx\int_{A^c_{\varepsilon}}\frac{|u_{\varepsilon}(x) - \breve{\eta}_{\varepsilon}(y)u_{\varepsilon}(y)|^2}{|x-y|^{N+2s}}dxdy\\
&\quad + \int_{A^c_{\varepsilon}}dx\int_{A_{\varepsilon}}\frac{|\breve{\eta}_{\varepsilon}(x)u_{\varepsilon}(x) - u_{\varepsilon}(y)|^2}{|x-y|^{N+2s}}dxdy + \int_{A^c_{\varepsilon}}dx\int_{A^c_{\varepsilon}}\frac{|\breve{\eta}_{\varepsilon}(x)u_{\varepsilon}(x) - \breve{\eta}_{\varepsilon}(y)u_{\varepsilon}(y)|^2}{|x-y|^{N+2s}}dxdy\\
&\le \int_{A_{\varepsilon}}dx\int_{A_{\varepsilon}}\frac{|u_{\varepsilon}(x) - u_{\varepsilon}(y)|^2}{|x-y|^{N+2s}}dxdy+ C\int_{A_{\varepsilon}}dx\int_{A^c_{\varepsilon}}\frac{|u_{\varepsilon}(x) - u_{\varepsilon}(y)|^2}{|x-y|^{N+2s}}dxdy\\
&\quad + C\int_{A_{\varepsilon}}dx\int_{A^c_{\varepsilon}}\frac{|(1 - \breve{\eta}_{\varepsilon}(y)) u_{\varepsilon}(y)|^2}{|x-y|^{N+2s}}dxdy + \int_{A^c_{\varepsilon}}dx\int_{A^c_{\varepsilon}}\frac{|\breve{\eta}_{\varepsilon}(x)u_{\varepsilon}(x) - \breve{\eta}_{\varepsilon}(y)u_{\varepsilon}(y)|^2}{|x-y|^{N+2s}}dxdy\\
&\le C\varepsilon^{N}  + C\int_{A_{\varepsilon}}dx\int_{A^c_{\varepsilon}}\frac{|(1 - \breve{\eta}_{\varepsilon}(y)) u_{\varepsilon}(y)|^2}{|x-y|^{N+2s}}dxdy\\
&\quad + \int_{A^c_{\varepsilon}}dx\int_{A^c_{\varepsilon}}\frac{|\breve{\eta}_{\varepsilon}(x)u_{\varepsilon}(x) - \breve{\eta}_{\varepsilon}(y)u_{\varepsilon}(y)|^2}{|x-y|^{N+2s}}dxdy\\
&\le C\varepsilon^{N}.
\end{align*}
As a result, we get \eqref{MMMeq2.29}.

The proof of \eqref{addeq2.27} is similar and we omit it.
\end{proof}

\section*{Acknowledgments}
The authors are grateful to the referee  for carefully reading the manuscript and for many valuable comments which largely improved the article. This work was partially supported by NSFC grants (No.12101150; No.11831009) and the Science and Technology Foundation of Guizhou Province ([2021]ZK008).

\section*{Conflict of interest}

The authors declare there is no conflicts of interest.


\begin{thebibliography}{99}



\bibitem{100}
 C. Alves,  O. Miyagaki,
 Existence and concentration of solution for a class of fractional elliptic equation in $\mathbb{R}^N$ via penalization method,
 \emph{Calc. Var. Partial Differential Equations}, \textbf{55}  (2016), 1-19.
 {https://doi.org/10.1007/s00526-016-0983-x}

\bibitem{5}
 A. Ambrosetti, M. Badiale,  S. Cingolani,
 {Semiclassical states of nonlinear Schr\"{o}dinger equations},
 \emph{Arch. Ration. Mech. Anal.}, \textbf{140}  (1997), 285--300 {https://doi.org/10.1007/s002050050067}

\bibitem{6}
 A. Ambrosetti,  A. Malchiodi,
 Perturbation methods and semilinear elliptic problems on $\mathbb{R}^N$,
 Progress in Mathmatics, vol. 240. Birfh$\ddot{a}$user Verlag, Basel (2006)
 {https://doi.org/10.1007/3-7643-7396-2}

\bibitem{14}
 A. Ambrosetti,  A. Malchiodi, W. M. Ni,
{Singularly perturbed elliptic equations with symmetry: existence of solutions concentrating on spheres},
 \emph{I. Comm. Math. Phys.,} \textbf{235} (2003), 427--466.
{https://doi.org/10.1007/s00220-003-0811-y}

\bibitem{AA}
 A. Ambrosetti, A. Malchiodi,
 Concentration phenomena for NLS: Recent results and new perspectives, perspectives in nonlinear partial differential equations,
{Contemp. Math.}, \textbf{446}  (2007), 19--30.
 {https://doi.org/10.1090/conm/446/08624}

%\bibitem{21}
%{Albert, J. P.,Bona, J. L., Saut, J.-C}: {Model equations for waves in stratified fluids}. Proc. Roy.
%Soc. London Ser. A \textbf{453}, 1233-1260 (1997)

\bibitem{APX}
 X. An, S. Peng,  C. Xie,
 {Semi-classical solutions for fractional Schr\"odinger equations with potential vanishing at infinity},
 \emph{J. Math. Phys.}, \textbf{60} (2019), 021501.
 {https://doi.org/10.1063/1.5037126}

\bibitem{ADP-AA-2021}
 X. An, L. Duan,  Y. Peng,
 {Semi-classical analysis for fractional Schr\"odinger equations with fast decaying potentials},
 \emph{Applicable Analysis},
 {https://doi.org/10.1080/00036811.2021.1880571}

%\bibitem{30}
%{Bonheure, D., Van Schaftingen, J.}: {Groundstates for the nonlinear Schr\"{o}dinger equation with potential
%vanishing at infinity}. Ann. Mat. Pura Appl. (4) \textbf{189}(2), 273-301 (2010)
\bibitem{BDP}
 T. Bartsch, E. N. Dancer, S. Peng,
 On multi-bump semiclassical bound states of nonlinear Schr\"{o}dinger euqations with electromagnetic fields,
 \emph{Adv. Differential Equations}, \textbf{7} (2006), 781--812.

\bibitem{16}
 D. Bonheure, S. Cingolani, M. Nys,
 {Nonlinear Schr\"{o}dinger equation: concentration on circles driven by an external maganetic field},
 \emph{Calc. Var. Partial Differential Equations}, \textbf{55} (2016), Article 82 33p
{https://doi.org/10.1007/s00526-016-1013-8}

\bibitem{J.B-L.J-DCDS-2007}
  J. Byeon, L. Jeanjean,
 Multi-peak standing waves for nonlinear Schr\"{o}dinger equations with a general nonlinearity,
  \emph{Discrete Contin. Dyn. Syst.}, \textbf{19} (2007), 255--269.
{https://doi.org/10.3934/dcds.2007.19.255}

\bibitem{CN}
 D. Cao,  E. S. Noussair,
 {Multi-bump standing waves with a critical frequency for nonlinear Schr\"{o}dinger equations},
 \emph{J. Differential Equations}, \textbf{203} (2004), 292--312.
 {https://doi.org/10.1016/j.jde.2004.05.003}

%\bibitem{18}
%{Cerami, G., Passseo, D., Solimini, S.}: {Infinitely many positive solutions to some scalar field equations with nonsymmetric coefficients}. Comm. Pure Appl. Math., \textbf{66} (2013), 372-413.

\bibitem{25}
 G. Chen,  Y. Zheng,
 {Concentration phenomena for fractional noninear Schr\"{o}dinger equations,}
 \emph{Commun. Pure Appl. Anal}, \textbf{13}  (2014), 2359--2376.
 {https://doi.org/10.3934/cpaa.2014.13.2359}

\bibitem{24}
 L. Caffarelli,  L. Silvestre,
 {An extension problem related to the fractional Laplacian},
 \emph{Comm. Partial Differential Equations}, \textbf{32}  (2007), 1245--1260.
 {https://doi.org/10.1080/03605300600987306}


\bibitem{7}
 S. Cingolani, M. Lazzo,
 {Multiple semiclassical standing waves for a class of nonlinear Schr\"{o}dinger
equations},
 \emph{Topol. Methods Nonlinear Anal.}, \textbf{10} (1997), 1--13.
  {https://doi.org/10.12775/TMNA.1997.019}

\bibitem{8}
 S. Cingolani, M. Lazzo,
 {Multiple positive solutions to nonlinear Schr\"{o}dinger equations with competing potential functions},
 \emph{J. Differential Equations}, \textbf{160} (2000) , 118--138.
 {https://doi.org/10.1006/jdeq.1999.3662}

\bibitem{CP}
  V. Coti-Zelati,   P. H. Rabinowitz,
 {Homoclinic orbits for a second order Hamiltonian systems possessing superquadratic potentials},
 \emph{J. Amer. Math. Soc.}, \textbf{4} (1991), 693-727.
 {https://doi.org/10.1090/S0894-0347-1991-1119200-3}

\bibitem{9}
 M. del Pino, P.L. Felmer,
 {Local mountain passes for semilinear elliptic problems in unbounded domains,}
 \emph{Calc. Var. Partial Differential Equations}, \textbf{4}  (1996), 121--137.
 {https://doi.org/10.1007/BF01189950}

\bibitem{10}
 M. del Pino, P. L. Felmer,
 {Semi-classical states for nonlinear Schr\"{o}dinger equations},
 \emph{J. Funct. Anal.}, \textbf{149} (1997), 245--265.
 {https://doi.org/10.1006/jfan.1996.3085}

\bibitem{MM3}
 M. del Pino, P. L. Felmer,
 {Multi-peak bound states for nonlinear Schr\"{o}dinger equations},
 \emph{Ann. Inst. H. Poincar\`{e}, Analyse non lin\`{e}aire}, \textbf{15}  (1998), 127--149.
 {https://doi.org/10.1016/s0294-1449(97)89296-7}


\bibitem{15}
 M. del Pino, M. Kowalczyk, J. Wei,
 {Concentration on curves for nonlinear Schr\"{o}dinger equations}.,
 \emph{Comm. Pure Appl. Math.}, \textbf{60} (2006), 113--146.
 {https://doi.org/10.1002/cpa.20135}



\bibitem{4}
 E. Di Nezza, G. Palatucci, E. Valdinoci,
 {Hitchhiker's guide to the fractional Sobolev spaces},
 \emph{Bull. Sci. Math.}, \textbf{136} (2012), 521--573.
 {https://doi.org/10.1016/j.bulsci.2011.12.004}

%\bibitem{1}
% P. Felmer, A. Quaas, J. Tan,
% {Positive solutions of the nonlinear Schr\"odinger equation with the fractional Laplacian},
% \emph{Proc. Roy. Soc. Edinburgh Sect. A}, \textbf{142}, 1237-1262 (2012)


\bibitem{11}
  A. Floer, A. Weinstein,
 {Nonspreading wave packets for the cubic Schr\"{o}dinger equation with a bounded potential},
 \emph{J. Funct. Anal.}, \textbf{69}  (1986), 397--408.
{https://doi.org/10.1016/0022-1236(86)90096-0}

\bibitem{19}
 R. L. Frank, E. Lenzmann,
 {Uniqueness of non-linear ground states for fractional Laplacians in $\mathbb{R}$},
 \emph{Acta. Math}, \textbf{210} (2013), 261--318.
 {https://doi.org/10.1007/s11511-013-0095-9}

\bibitem{20}
 R. L. Frank, E. Lenzmann, L. Silvestre,
 {Uniqueness of radial solutions for the fractional Laplacians},
 \emph{Comm. Pure. Appl. Math.}, \textbf{69} (2016),  1671--1726.
 {https://doi.org/10.1002/cpa.21591}

\bibitem{FS}
 R. L. Frank, R. Seiringer,
 {Non-linear ground state representations and sharp Hardy inequalities},
 \emph{J. Funct. Anal.}, \textbf{255} (2008), 3407--3430.
 {https://doi.org/10.1016/j.jfa.2008.05.015}

\bibitem{26}
 M. M. Fall, F. Mahmoudi, E. Valdinoci,
 {Ground states and concentration phenomena for the fractional Schr\"{o}dinger equation},
 \emph{Nonlinearity}, \textbf{28} (2015), 1937--1961.
 {https://doi.org/10.1088/0951-7715/28/6/1937}

\bibitem{FQT}
 P. Felmer,  A. Quaas,  J. Tan,
 Positive solutions of the nonlinear Schr\"odinger equation with the fractional Laplacian,
 \emph{Proc. Roy. Soc. Edinburgh Sect. A,} \textbf{142} (2012), 1237--1262.
 {https://doi.org/10.1017/S0308210511000746}

\bibitem{T.Hu-W.Shuai-JDE-2018}
 T. Hu,  W. Shuai,
 Multi-peak solutions to Kirchhoff equations in $\R^3$ with general nonlinearlity,
 \emph{J. Differential Equations,} \textbf{265} (2018), 3587--3617.
 {https://doi.org/10.1016/j.jde.2018.05.012}

\bibitem{3}
 {N. Laskin,}
{Fractional quantum mechanics and Levy path integrals},
\emph{Phys. Lett. A}, \textbf{268} (2000) , 298--305.
 {https://doi.org/10.1016/S0375-9601(00)00201-2}

\bibitem{2}
{N. Laskin, }
 {Fractional Schr\"{o}dinger equation},
 \emph{Phys. Rev. E,} \textbf{66} (2002), 056108.
 {https://doi.org/10.1103/PhysRevE.66.056108}

%\bibitem{23}
%{Ma, L., Zhao, L}. {Classification of positive solitary solutions of the nonlinear Choquard equation}.
%Arch. Ration. Mech. Anal. \textbf{195}, no. 2, 455-467 (2010)

%\bibitem{28}
% {V. Moroz, J.  Van Schaftingen, }
%{Semi-classical states for the Choquard equation},
% \emph{Calc. Var. Partial Differential Equations}, \textbf{52}(199-235), (2015) {https://doi.org/10.1007/s00526-014-0709-x}

\bibitem{12}
 {Y. G. Oh, }
 {Existence of semiclassical bound states of nonlinear Schr\"{o}dinger equations with potentials of the class $(V)_a$},
 \emph{Comm. Partial Differential Equations,} \textbf{13} (1988), 1499--1519. {https://doi.org/10.1080/03605308808820585}





\bibitem{13}
 {P. H. Rabinowitz, }
 {On a class of nonlinear Schr\"{o}dinger equations},
 \emph{Z. Angew. Math. Phys.}, \textbf{43} (1992), 270--291. {https://doi.org/10.1007/BF00946631}

\bibitem{S}
 S. Secchi,
 Ground state solutions for nonlinear fractional Schr\"odinger equations in $\mathbb{R}^N$.
 \emph{J. Math. Phys.}, \textbf{54} (2013), 031501. {https://doi.org/10.1063/1.4793990}



\bibitem{27}
 X. Shang, J. Zhang,
 {Concentrating solutions of nonlinear fractional Schr\"{o}dinger equation with potentials},
 \emph{ J. Differential Equations},  \textbf{258} (2015), 1106--1128.
 {https://doi.org/10.1016/j.jde.2014.10.012}


%\bibitem{29}
%{Willem M.}: {Analyse harmonique r\'{e}elle,} Hermann, Paris, (1995)

\bibitem{MW}
 M. Willem,
 {Minimax theorems},
 Birkh\"{a}user, 1996.


%\bibitem{22}
%{Weinstein, M. I.}: {Existence and dynamic stability of solitary wave solutions of equations arising
%in long wave propagation}. Comm. Partial Differential Equations \textbf{12}, no. 10, 1133-1173 (1987)

\bibitem{17}
 J. Wei, S. Yan,
 {Infinitely many positive solutions for the nonlinear Schr\"{o}dinger equations in $\mathbb{R}^N$},
 \emph{Calc. Var. Partial Differential Equations},  \textbf{37}  (2010), 423--439. {https://doi.org/10.1007/s00526-009-0270-1}

\bibitem{M.I.-Weinstein-CPDE-1987}
 M. I. Weinstein,
 {Existence and dynamic stability of solitary wave solutions of equations arising in long wave propagation},
 \emph{Comm. Partial Differential Equations}, \textbf{12} (1987), 1133--1173.
{https://doi.org/10.1080/03605308708820522}



\end{thebibliography}
\end{document}